\documentclass[12pt]{amsart}%[11pt]{article}

\usepackage{amsmath}
\usepackage{latexsym}
\usepackage{amssymb}
\usepackage{amsfonts}
\usepackage{multirow}
\usepackage{graphics}
\usepackage{hyperref}

\newcommand\di{\,\big|\,} 
\newcommand\Hol{\mathrm{Hol}} \newcommand\Rad{\mathrm{Rad}} \newcommand\Soc{\mathrm{Soc}}
\newcommand\calB{{\mathcal B}}  \newcommand\calT{{\mathcal T}}
 \newcommand\G{\mathrm{G}}
\newcommand\D{\mathrm{D}} \newcommand\Q{\mathrm{Q}}
 \newcommand\M{\mathrm{M}} \newcommand\Z{\mathrm{C}}
\newcommand\Pa{\mathrm{P}}
\newcommand\N{\mathrm{N}}
\newcommand\X{\mathrm{X}}
\newcommand\C{\mathbf{C}}

\newcommand\Aut{\mathrm{Aut}}  \newcommand\Out{\mathrm{Out}}

\newcommand\Ome{\Omega}

\newcommand\Del{\Delta}
\newcommand\GF{\mathrm{GF}}
 \newcommand\POm{\mathrm{P}{\Omega}} \newcommand\PSO{\mathrm{PSO}}
 \newcommand\PSp{\mathrm{PSp}}
 \newcommand\GaL{\mathrm{\Gamma L}}
\newcommand\AGammaL{\mathrm{A\Gamma L}}
 \newcommand\PGaL{\mathrm{P\Gamma L}} 
\newcommand\A{\mathrm{A}} \newcommand\Sym{\mathrm{Sym}} \newcommand\Sy{\mathrm{S}}
\newcommand\PSL{\mathrm{PSL}} \newcommand\PGL{\mathrm{PGL}}
\newcommand\GL{\mathrm{GL}} \newcommand\SL{\mathrm{SL}}
\newcommand\AGL{\mathrm{AGL}}
\newcommand\PSU{\mathrm{PSU}}  \newcommand\GU{\mathrm{GU}}
\newcommand\Sz{\mathrm{Sz}}

\newcommand\magma{\textsc{Magma}}

\newtheorem{remark}{Remark}[section]
\newtheorem{theorem}{Theorem}[section]%
\newtheorem{lemma}[theorem]{Lemma}%

\begin{document}

\title[Metacyclic transitive subgroups]
{Finite quasiprimitive permutation groups with a metacyclic transitive subgroup}

\thanks{This work was partially supported by NSFC Grants (11231008, 11461007, 11771200)}

\author{Cai Heng Li}
\address{(Li) Department of Mathematics, Southern University of Science and Technology\\Shenzhen, Guangdong 518055\\P. R. China}
\email{lich@sustc.edu.cn}

\author{Jiangmin Pan}
\address{(Pan) School of Statistics and Mathematics\\Yunnan University of Finance and Economics\\Kunming, Yunnan 650221\\P. R. China}
\email{jmpan@ynu.edu.cn}

\author{Binzhou Xia}
\address{(Xia) School of Mathematics and Statistics\\The University of Melbourne\\Parkville, VIC 3010\\Australia}
\email{binzhoux@unimelb.edu.au}
\maketitle

\begin{abstract}
In this paper, we classify finite quasiprimitive permutation groups with a metacyclic transitive subgroup, solving a problem initiated by Wielandt in 1949. It also involves the classification of factorizations of almost simple groups with a metacyclic factor.


\end{abstract}

\section{Introduction}

A transitive permutation group $G\leqslant\Sym(\Ome)$ is called {\it primitive} if $\Ome$ has no nontrivial $G$-invariant partition. The study of primitive permutation groups containing a certain transitive subgroup played an important role in the history of permutation group theory, for which the reader is referred to Problem~3 in the excellent survey \cite{Neumman} of Neumann or the textbooks~\cite{DM-book,Wielandt}. In 1949, Wielandt \cite{Wielandt1949} proved that finite primitive permutation groups containing a regular dihedral group are $2$-transitive and initiated the following problem.

\vskip0.1in
\noindent{\bf Problem A.} Classify finite primitive permutation groups which contain a regular
metacyclic subgroup.
\vskip0.1in

Recall a group $R$ is called {\it metacyclic} if there is a normal subgroup $N$ of $R$ such that both $N$ and $R/N$ are cyclic.
Other partial results for Problem A are known for cyclic subgroups
(Burnside \cite{Burnside} and Schur \cite{Schur}),
a family of metacyclic subgroups of order $4n$ for positive integers $n$ (Scott \cite{Scott}),
metacyclic subgroups of order $3p$ for certain special primes $p$
(Nagai \cite{Nagai}), metacyclic subgroups of order
$p^3$ with $p$ prime (Jones \cite{Jones72})
and metacyclic subgroups of square-free order (the first author and Seress \cite{LS05}).

The main purpose of this paper is to solve Problem A in a more general version,
that is, classify finite quasiprimitive permutation groups with a metacyclic transitive subgroup.
Recall that a permutation group is called {\it quasiprimitive} if each of its non-trivial normal subgroups
is transitive.
It is easily known that primitive permutation groups are quasiprimitive.
By the O'Nan-Scott-Praeger theorem, see \cite[Section~5]{q-gps},
finite quasiprimitive permutation groups can be divided into eight O'Nan-Scott types.

Primitive permutation groups with a certain transitive subgroup
also play a key role in the study of many problems in algebraic combinatorics,
especially in the study of Cayley graphs and Cayley maps.
In particular, one motivation of this work is the effort to classify edge-transitive metacirculants, an important class of symmetric graphs. A graph is called a (weak) {\it metacirculant} if it has a group of automorphisms which is vertex-transitive and metacyclic, see \cite{Song,MS}. The classification given in Theorem \ref{Quasiprimitive} below provides the foundation for the study of metacirculants. 
%In \cite{Li-Lu-Pan}, the special case of Theorem \ref{Quasiprimitive} when $G$ is primitive is successfully applied to solve problems regarding edge-transitive metacirculants.

\begin{theorem}\label{Quasiprimitive}
Let $G\leqslant\Sym(\Ome)$ be a finite quasiprimitive permutation group, $\omega\in\Ome$ and $R$ be a metacyclic transitive subgroup of $G$. Then one of the following holds.
\begin{itemize}
\item[(a)] $G$ is an affine primitive group of dimension at most $4$, and $(G,R,G_\omega)$ is as described in \emph{Theorem~\ref{HA}}.

\item[(b)] $G$ is an almost simple group, and $(G,R,G_\omega)$ is as described in \emph{Theorem~\ref{AS}} $($as $(G,A,B)$ there$)$.
% and $G=\Soc(G)B$.

\item[(c)] $G$ is primitive of holomorph simple or simple diagonal type, and $(G,R)$ is as described in \emph{Theorem~\ref{Diagonal}}.

\item[(d)] $G$ is of product action type, and $(G,R)$ is as described in \emph{Theorem~\ref{PA}}.
\end{itemize}
\end{theorem}

%The proof of Theorem~\ref{Quasiprimitive} depends on Liebeck-Praeger-Saxl's classification of maximal factorizations of almost simple groups in \cite{BOOK}, and so depends on the classification of finite simple groups. It is also based on the O'Nan-Scott-Praeger theorem for quasiprimitive groups, proved in \cite{Praeger-qp}.

A {\it factorization} of a group $G$ is an expression of the group as a product of two proper subgroups which are called {\it factors} of $G$. It is easy to see that a permutation group $G$ contains a transitive subgroup $H$ if and only if $G=HG_\omega$, where $G_\omega$ is a point stabilizer.
Determining all the factorizations of almost simple groups with one factor metacyclic is a crucial step for the proof of Theorem~\ref{Quasiprimitive}. This will be done in Section 3
based on Liebeck-Praeger-Saxl's classification of maximal factorizations of almost simple groups in \cite{BOOK}.
Section 4 is devoted to affine groups, and in Section 5 we will determine the holomorph simple and simple diagonal types which are sometimes included together as the {\it diagonal type}. Then in the last section, the other O'Nan-Scott types will be dealt with and Theorems \ref{Quasiprimitive} will be proved.

\section{Preliminaries}\label{sec1}

In this section, we set up terminology and technical lemmas. All groups in this paper are supposed to be finite if there is no special instruction. Some of our notations will follow \cite{atlas} and \cite{BOOK}.

For a permutation group $G\leqslant\Sym(\Ome)$ and a subset $\Del\subseteq\Ome$, let $G_\Del$ be the setwise stabilizer of $\Del$.
If $H\leqslant G$ fixes $\Del$ setwise, then let $H^\Del$ be the permutation group induced by $H$ on $\Del$. In particular, $G_\Del^\Del$ denotes the permutation group induced by $G_\Del$ on $\Del$. Let $G_{(\Del)}$ be the pointwise stabilizer of $\Del$ in $G$. Then one immediately has $G_\Del^\Del\cong G_\Del/G_{(\Del)}$. For a $G$-invariant partition $\calB$ of $\Ome$, let $G^\calB$ be the permutation group induced by $G$ on $\calB$.
Recall that a subgroup $H\leqslant G$ is called {\it core-free} in $G$ if $H$ does not contain any nontrivial normal subgroup of $G$.

\begin{lemma}\label{NormalCyclic}
Let $G\leqslant\Sym(\Ome)$ be a transitive permutation group
and $H$ be a cyclic normal subgroup of $G$. Then $H$ is semiregular on $\Ome$.
\end{lemma}

\begin{proof}
Note that $G_\omega$ is core-free in $G$ for all $\omega\in\Ome$. Since $H$ is cyclic, $H\cap G_\omega$ is a characteristic subgroup of $H$, and thus normal in $G$. Therefore, $H_\omega=H\cap G_\omega=1$, which means that $H$ is semiregular on $\Ome$.
\end{proof}

Let $H\leqslant\Sym(\Del)$ and $k$ be a positive integer. The wreath product $H\wr\Sy_k$ has an action on $\Del^k$ defined by
\[
(\delta_1,\dots,\delta_k)^{(h_1,\dots,h_k,\sigma^{-1})}
=((\delta_{1^\sigma})^{h_{1^\sigma}},\dots,(\delta_{k^\sigma})^{h_{k^\sigma}}),
\]
where $(\delta_1,\dots,\delta_k)\in\Del^k$, $(h_1,\dots,h_k)\in H^k$ and $\sigma\in\Sy_k$.
The permutation group $H\wr\Sy_k$ acting on $\Del^k$ as above is called the {\it primitive wreath product},
and the normal subgroup $H^k$ is called the {\it base group}.

\begin{lemma}\label{SemiregCyclic}
Let $G=\Sy_n\wr\Sy_k$ be a primitive wreath product with the base group $K=\Sy_n^k$.
If $H$ is a semiregular cyclic subgroup of $K$, then $|H|$ divides $n$.
\end{lemma}

\begin{proof}
Note that $H$ is the direct product of its Sylow subgroups. We only need to prove the lemma
for the case that $H=\langle a\rangle$ is a cyclic group of order $p^\ell$, where $p$ is a prime and $\ell$ is a positive integer.
Assume without loss of generality that $a=(a_1,\dots,a_k)\in K$ with $o(a_1)=\dots=o(a_j)=p^\ell>o(a_i)$ for $i=j+1,\dots,k$.

Suppose that none of $\langle a_1\rangle,\dots,\langle a_j\rangle$ is semiregular in $\Sy_n$.
Then for each $i\in\{1,\dots,j\}$, the element $a_i^{p^{\ell-1}}$  fixes some point $\delta_i$. Hence $(a_1^{p^{\ell-1}},\dots,a_j^{p^{\ell-1}})$ has a fixed point $(\delta_1,\dots,\delta_j)$. Since
\[
a^{p^{\ell-1}}=(a_1^{p^{\ell-1}},\dots,a_k^{p^{\ell-1}})=(a_1^{p^{\ell-1}},\dots,a_j^{p^{\ell-1}},1,\dots,1),
\]
we conclude that $a^{p^{\ell-1}}$ has a fixed point, contradicting the condition that $H$ is semiregular.

Therefore, at least one of $\langle a_1\rangle,\dots,\langle a_j\rangle$ is semiregular in $\Sy_n$, say $\langle a_1\rangle$. It follows that $|H|=p^\ell=o(a_1)$ divides $n$.
\end{proof}

\begin{lemma}\label{DirectProduct}
Let $p$ be a prime, and $G=H^k$ with $k\geqslant2$. Then each metacyclic $p$-subgroup of $G$ is isomorphic to a subgroup of $H^2$. As a consequence, each metacyclic subgroup of $G$ has order dividing $|H|^2$.
\end{lemma}

\begin{proof}
The conclusion is trivial for $k=2$. Thus we assume $k\geqslant3$.
Let $G=H_1\times\dots\times H_k$ with $H_1\cong\dots\cong H_k\cong H$,
and let $P$ be a metacyclic $p$-subgroup of $G$.
If $P\cap H_i\neq1$ for each $i\in\{1,\dots,k\}$,
then $P\cap H_i\geqslant\Z_p$ and so
\[
P\geqslant(P\cap H_1)\times\dots\times(P\cap H_k)
\]
has an elementary abelian $p$-subgroup $\Z_p^k$, contradicting that $P$ is metacyclic.
Hence there exists $i\in\{1,2,\dots,k\}$ such that $P\cap H_i=1$. It follows that
\[P\cong(PH_i)/H_i\leqslant G/H_i\cong H^{k-1}.\]
Thereby we derive the conclusion of the lemma by induction.
\end{proof}

Note that the group of unitriangular matrices is a Sylow $p$-subgroup of $\GL_n(p^f)$. We have the following result on the largest order of $p$-elements in $\GL_n(p^f)$.
%, which is easy to know once we write the $p$-element in $\GL_n(p^f)$ as a unitriangular $n\times n$ matrix.

\begin{lemma}\label{Exponent}
Let $p$ be a prime. The largest order of $p$-elements in $\GL_n(p^f)$ equals the smallest $p$-power which is greater than or equal to $n$, that is, $p^{\lceil\log_pn\rceil}$.
\end{lemma}

The next lemma is useful in the study of factorizations of almost simple groups.

\begin{lemma}\label{Order}
Let $A,B$ be subgroups of $G$ and $L$ be a normal subgroup of $G$. If $G=AB$, then $|A\cap L||B\cap L||G/L|$ is divisible by $|L|$.
\end{lemma}

\begin{proof}
It follows from the factorization $G=AB$ that $|G|$ divides $|A||B|$. Then since $|A|$ divides $|A\cap L||G/L|$ and $|B|$ divides $|B\cap L||G/L|$, we conclude that $|G|$ divides $|A\cap L||B\cap L||G/L|^2$. Consequently, $|A\cap L||B\cap L||G/L|$ is divisible by $|L|$.
\end{proof}

Let $a$ and $m$ be positive integers. A prime number $r$ is called a \emph{primitive prime divisor} of $a^m-1$ if $r$ divides $a^m-1$ but does not divide $a^\ell-1$ for any positive integer $\ell<m$.
If $r$ is a primitive prime divisor of $a^m-1$, then $m\di r-1$, and in particular $r>m$.
The following Zsigmondy's theorem is on the existence of primitive prime divisors.

\begin{theorem}
\emph{(Zsigmondy, see \cite[Theorem IX.8.3]{blackburn1982finite})} Let $a$ and $m$ be positive integers. Then $a^m-1$ has a primitive prime divisor except for $(a,m)=(2,6)$ or $(2^k-1,2)$ for some positive integer $k$.
\end{theorem}

\section{Almost simple type}

In this section, we classify factorizations of almost simple groups with a metacyclic factor. This extends the classification of factorizations of almost simple groups with a cyclic factor (given in \cite{Jones} and \cite{Abel-B-gps} independently) and those with a dihedral factor \cite{rotary}.

We first discuss factorizations of almost simple groups $G$ with socle $L=\PSL_2(q)$, and assume $q\geqslant7$ and $q\neq9$ as the case $q=4,5$ or $9$ is attributed to alternating groups.

\begin{theorem}\label{2-dimension}
Let $L=\PSL_2(q)\leqslant G\leqslant\PGaL_2(q)$, $7\leqslant q=p^f\ne 9$ for a prime number $p$,  and $G=AB$ for
core-free subgroups $A$ and $B$ of $G$. Then interchanging $A$ and $B$ if necessary, exactly one of the following holds,
where $d=\gcd(2,q-1)$.
\begin{itemize}
\item[(a)] $A\cap L\leqslant\D_{2(q+1)/d}$ and $B\cap L\leqslant\Z_p^f{:}\Z_{(q-1)/d}$.
\item[(b)] $(G,A,B)$ lies in \emph{Table} $\ref{tab1}$.
\end{itemize}
\end{theorem}

\begin{table}[htbp]
\caption{}\label{tab1}
\centering
\begin{tabular}{|l|l|l|l|}
\hline
row & $G$ & $A$ & $B$\\
\hline
1 & $\PSL_2(7)$ & $\Z_7$, $\Z_7{:}\Z_3$ & $\Sy_4$\\
\hline
2 & $\PGL_2(7)$ & $\Z_7{:}\Z_2$, $\Z_7{:}\Z_6$ & $\Sy_4$\\
\hline
3 & $\PSL_2(11)$ & $\Z_{11}{:}\Z_5$ & $\A_4$\\
\hline
4 & $\PSL_2(11)$ & $\Z_{11}$, $\Z_{11}{:}\Z_5$ & $\A_5$\\
\hline
5 & $\PGL_2(11)$ & $\Z_{11}{:}\Z_{10}$ & $\A_4$\\
\hline
6 & $\PGL_2(11)$ & $\Z_{11}{:}\Z_5$, $\Z_{11}{:}\Z_{10}$ & $\Sy_4$\\
\hline
7 & $\PGL_2(11)$ & $\Z_{11}{:}\Z_2$, $\Z_{11}{:}\Z_{10}$ & $\A_5$\\
\hline
8 & $\PGaL_2(16)$ & $\Z_{17}{:}\Z_8$ & $(\A_5\times\Z_2).\Z_2$\\
\hline
9 & $\PSL_2(19)$ & $\Z_{19}{:}\Z_9$ & $\A_5$\\
\hline
10 & $\PGL_2(19)$ & $\Z_{19}{:}\Z_{18}$ & $\A_5$\\
\hline
11 & $\PSL_2(23)$ & $\Z_{23}{:}\Z_{11}$ & $\Sy_4$\\
\hline
12 & $\PGL_2(23)$ & $\Z_{23}{:}\Z_{22}$ & $\Sy_4$\\
\hline
13 & $\PSL_2(29)$ & $\Z_{29}{:}\Z_7$, $\Z_{29}{:}\Z_{14}$ & $\A_5$\\
\hline
14 & $\PGL_2(29)$ & $\Z_{29}{:}\Z_{28}$ & $\A_5$\\
\hline
15 & $\PSL_2(59)$ & $\Z_{59}{:}\Z_{29}$ & $\A_5$\\
\hline
16 & $\PGL_2(59)$ & $\Z_{59}{:}\Z_{58}$ & $\A_5$\\
\hline
\end{tabular}
\end{table}

\begin{proof}
By Lemma \ref{Order} we know that
\begin{equation}\label{eq4}
d^2f|A\cap L||B\cap L|\text{ is divisible by }q(q^2-1).
\end{equation}
By a result of Dickson \cite{huppert1967endliche}, every subgroup of $L$ is contained in one of the following groups:
\begin{eqnarray}
&&\Z_p^f{:}\Z_{(q-1)/d},\ \D_{2(q+1)/d},\ \D_{2(q-1)/d},\ \PSL_2(p^m)\mbox{ with }m\di f,\ \PGL_2(p^m)\mbox{ with }2m\di f,\label{class1}\\
&&\A_4\mbox{ with }q\neq2^{2k+1},\ \Sy_4\mbox{ with }q^2\equiv1\pmod{16},\ \A_5\mbox{ with }q^2\not\equiv4\pmod{5}.\label{class2}
\end{eqnarray}
In view of (\ref{eq4}) it is easy to see that $A\cap L$ and $B\cap L$ cannot be contained in groups in (\ref{class2}) simultaneously.

If $\PSL_2(16)\leqslant G\leqslant\PSL_2(16){:}\Z_2$, then Lemma \ref{Order} implies that $A\cap L\leqslant\Z_2^4{:}\Z_{15}$ and $B\cap L\leqslant\D_{34}$. If $G=\PGaL_2(16)$, then Lemma \ref{Order} implies that either $A\cap L\leqslant\Z_2^4{:}\Z_{15}$ and $B\cap L\leqslant\Z_{17}{:}\Z_2$, or $A=(\A_5\times\Z_2).\Z_2$ and $B=\Z_{17}{:}\Z_8$. Thus the theorem holds for $q=16$, and we assume $q\neq16$ in the rest of the proof.

Suppose that both $A\cap L$ and $B\cap L$ are contained in groups in (\ref{class1}). If $p^{2f}-1$ has a primitive prime divisor $r$, then $r>2f$ and the only group in (\ref{class1}) with order divisible by $r$ is $\D_{2(q+1)/d}$, whence we may let $A\cap L\leqslant\D_{2(q+1)/d}$. Further viewing that $p^f$ divides $f|B\cap L|$, we get $B\cap L\leqslant\Z_p^f{:}\Z_{(q-1)/d}$ as $q\neq16$. If $p^{2f}-1$ does not have a primitive prime divisor, then $f=1$ or $(p,f)=(2,3)$ by Zsigmondy's theorem, in which case we still have $A\cap L\leqslant\D_{2(q+1)/d}$ and $B\cap L\leqslant\Z_p^f{:}\Z_{(q-1)/d}$ as a consequence of (\ref{eq4}).

Next suppose that $A\cap L$ is contained in a group in (\ref{class1}) and $B\cap L$ is contained in a group in (\ref{class2}). According to (\ref{eq4}) we have $A\cap L\leqslant\Z_p^f{:}\Z_{(q-1)/d}$ as $q\neq16$. Now (\ref{eq4}) implies that $df|B\cap L|$ is divisible by $q+1$. In particular, $(q+1)\di2f\cdot120$. If $f\geqslant3$, then either $p^{2f}-1$ has a primitive prime divisor $r$ or $(p,f)=(2,3)$. Since $(q+1)\di240f$, the latter cannot occur. For the former, $r$ divides $q+1$ but does not divide $240f$ as $r>2f\geqslant6$, a contradiction. Thereby $f=1$ or $2$, and then we conclude that one of the following occurs: $q=7$ and $B\cap L\leqslant\Sy_4$; $q=11$ and $B\cap L=\A_4$ or $\A_5$; $q=19$ and $B\cap L\leqslant\A_5$; $q=23$ and $B\cap L\leqslant\Sy_4$; $q=29$ and $B\cap L\leqslant\A_5$; $G=\PGL_2(47)$ and $B\cap L\leqslant\Sy_4$; $q=59$ and $B\cap L\leqslant\A_5$.

(i). $q=7$ and $B\cap L\leqslant\Sy_4$. If $G=\PSL_2(7)$, then $\Z_7\leqslant A\leqslant\Z_7{:}\Z_3$ and $B=\Sy_4$, as in row 1 of Table \ref{tab1}. If $B=\PGL_2(7)$, then $\Z_7{:}\Z_2\leqslant A\leqslant\Z_7{:}\Z_6$ and $B=\Sy_4$ as in row 2 of Table \ref{tab1}.

(ii). $q=11$ and $B\cap L=\A_4$ or $\A_5$. If $G=\PSL_2(11)$, then $A=\Z_{11}{:}\Z_5$ and $B=\A_4$ or $\Z_{11}\leqslant A\leqslant\Z_{11}{:}\Z_5$ and $B=\A_5$, as in row 3 or row 4 of Table \ref{tab1}, respectively. If $G=\PGL_2(11)$, then $B=\A_4$, $\Sy_4$ or $\A_5$ and an easy case by case analysis leads to rows 5--7 of Table \ref{tab1}.

(iii). $q=19$ and $B\cap L\leqslant\A_5$. If $G=\PSL_2(19)$, then $A=\Z_{19}{:}\Z_9$ and $B=\A_5$, as in row 9 of Table \ref{tab1}. If $G=\PGL_2(19)$, then $A=\Z_{19}{:}\Z_{18}$ and $B=\A_5$, as in row 10 of Table \ref{tab1}.

(iv). $q=23$ and $B\cap L\leqslant\Sy_4$. If $G=\PSL_2(23)$, then $A=\Z_{23}{:}\Z_{11}$ and $B=\Sy_4$, as in row 11 of Table \ref{tab1}. If $G=\PGL_2(23)$, then $A=\Z_{23}{:}\Z_{22}$ and $B=\Sy_4$, as in row 12 of Table \ref{tab1}.

(v). $q=29$ and $B\cap L\leqslant\A_5$. If $G=\PSL_2(29)$, then $\Z_{29}{:}\Z_7\leqslant A\leqslant\Z_{29}{:}\Z_{14}$ and $B=\A_5$, as in row 13 of Table \ref{tab1}. If $G=\PGL_2(29)$, then $A=\Z_{29}{:}\Z_{28}$ and $B=\A_5$, as in row 14 of Table \ref{tab1}.

(vi). $q=47$, $G=\PGL_2(47)$ and $B\cap L\leqslant\Sy_4$. In this case, $A\leqslant\Z_{47}{:}\Z_{46}$ and $B\leqslant\Sy_4$, and so $|A||B|$ divides $47\cdot46\cdot24=|G|/2$, a contradiction.

(vii). $q=59$ and $B\cap L\leqslant\A_5$. If $G=\PSL_2(59)$, then $A=\Z_{59}{:}\Z_{29}$ and $B=\A_5$, as in row 15 of Table \ref{tab1}. If $G=\PGL_2(59)$, then $A=\Z_{59}{:}\Z_{58}$ and $B=\A_5$, as in row 16 of Table \ref{tab1}.
\end{proof}

We introduce some notation for the following result. Let $T$ be a classical linear group on $V$ with center $Z$ such that $T/Z$ is a classical simple group, and $X$ be a subgroup of $\GL(V)$ containing $T$ as a normal subgroup. Then for any subgroup $Y$ of $X$, denote by $\hat{\,}Y$ the subgroup $(Y\cap T)Z/Z$ of $T/Z$. For the definition of the subgroups $\Pa_i$ and $\N_i$ of classical groups, see \cite[2.2.4]{BOOK}. If $q$ is a prime power, then we denote the elementary abelian group of order $q$ simply by $q$ when there is no confusion.

\begin{lemma}\label{LowerLinear}
Let $G$ be an almost simple group with socle $L=\PSL_n(q)$. Suppose $G=AB$ for subgroups $A,B$ of $G$ such that $A\cap L\leqslant\hat{\,}\GL_1(q^n){:}n$ and $B\cap L\leqslant\Pa_1$ or $\Pa_{n-1}$. If $A$ is metacyclic and $(n,q)\neq(2,4)$, $(3,2)$ or $(4,2)$, then one of the following holds.
\begin{itemize}
\item[(a)] $q^{n-1}{:}\SL_{n-1}(q)\trianglelefteq B\cap L$.
\item[(b)] $(G,A,B)$ lies in \emph{Table} $\ref{tab2}$.
\end{itemize}
\end{lemma}

\begin{table}[htbp]
\caption{}\label{tab2}
\centering
\begin{tabular}{|l|l|l|l|}
\hline
row & $G$ & $A$ & $B$\\
\hline
1 & $\PSL_3(3)$ & $\Z_{13}{:}\Z_3$ & $\AGammaL_1(9)$ \\
\hline
2 & $\PSL_3(3).\Z_2$ & $\Z_{13}{:}\Z_6$ & $\AGammaL_1(9)$ \\
\hline
3 & $\PSL_3(4).\Sy_3$ & $(\Z_7{:}\Z_3)\times\Sy_3$ & $\Z_2^4{:}(\Z_3\times\D_{10}).\Z_2$ \\
\hline
4 & $\PSL_3(8).\Z_3$ & $\Z_{73}{:}\Z_9$ & $(\Z_2^3.\Z_2^6){:}\Z_7^2{:}\Z_3$ \\
\hline
5 & $\PSL_3(8).\Z_6$ & $\Z_{73}{:}\Z_9$ & $(\Z_2^3.\Z_2^6){:}\Z_7^2{:}\Z_6$ \\
\hline
6 & $\PSL_3(8).\Z_6$ & $\Z_{73}{:}\Z_{18}$ & $(\Z_2^3.\Z_2^6){:}\Z_7^2{:}\Z_3$ \\
\hline
7 & $\PSL_3(8).\Z_6$ & $\Z_{73}{:}\Z_{18}$ & $(\Z_2^3.\Z_2^6){:}\Z_7^2{:}\Z_6$ \\
\hline
\end{tabular}
\end{table}

\begin{proof}
Let $q=p^f$ with prime number $p$. Note that if $n=2$ then part~(a) turns out to be $\Z_p^f\trianglelefteq B\cap L$. For $(n,q)\in\{(3,3),(3,4),(3,5),(3,8),(3,9)\}$ computation in \magma~\cite{magma} directly verifies the lemma. Thus we assume $(n,q)\neq(3,3)$, $(3,4)$, $(3,5)$, $(3,8)$ or $(3,9)$ in the following.

\textbf{Case 1.} Assume $n=2$. Then $q\geqslant5$ and $B\cap L\leqslant\Z_p^f{:}\Z_{(q-1)/\gcd(2,q-1)}$.

First suppose $f\leqslant2$. Then $p>2$ as $q\geqslant5$. We derive from Lemma \ref{Order} that $q$ divides $|B\cap L|$. Hence $q{:}\SL_1(q)=\Z_p^f\trianglelefteq B\cap L$ as in part~(a).

Next suppose $f\geqslant3$. By Zsigmondy's theorem, $p^f-1$ has a primitive prime divisor $r$ except $(p,f)=(2,6)$. If $(p,f)=(2,6)$, then $2^4\cdot3\cdot7$ divides $|B\cap L|$ by Lemma \ref{Order}, but the only subgroups of $\Z_2^6{:}\Z_{63}$ in $\PSL_2(64)$ with order divisible by $2^4\cdot3\cdot7$ are $\Z_2^6{:}\Z_{21}$ and $\Z_2^6{:}\Z_{63}$, which leads to $\Z_2^6\trianglelefteq B\cap L$. If $(p,f)\neq(2,6)$, then Lemma \ref{Order} implies that $B\cap L$ has order divisible by $r$ and thus has an element of order $r$. Note that $\Z_r$ does not have any faithful representation over $\GF(p)$ of dimension less than $f$. We then have $\Z_p^f{:}\Z_r\trianglelefteq B\cap L$ and so part~(a) holds.

\textbf{Case 2.} Assume $n\geqslant3$ and assume without loss of generality that $B\cap L\leqslant\Pa_1$. Let $M=\Pa_1$ be a maximal subgroup of $L$ containing $B$. From $G=AB$ we derive that $|G|$ divides $|A||B|$ and thus $|L|$ divides $|A||B\cap L|$. Then as $|A|$ divides $2nf|L|/|M|$, it follows that $|M|/|B\cap L|$ divides $2nf$. Note that $M$ has a unique unsolvable composition factor $\PSL_{n-1}(q)$, and $M^{(\infty)}=q^{n-1}{:}\SL_{n-1}(q)$. Let $R=\Rad(M)$, $\overline{M}=M/R$ and $\overline{K}=(B\cap L)R/R$. Then $\overline{M}$ is an almost simple group with socle $\PSL_{n-1}(q)$ and  $|\overline{M}|/|\overline{K}|$ divides $|M|/|B\cap L|$ and thus divides $2nf$. Moreover, we conclude from~\cite{cooperstein1978minimal} that each proper subgroup of $\Soc(\overline{M})=\PSL_{n-1}(q)$ has index greater than $2nf$. If $\overline{K}$ is core-free in $\overline{M}$, then
\[
\frac{|\Soc(\overline{M})|}{|\overline{K}\cap\Soc(\overline{M})|}
=\frac{|\overline{K}\,\Soc(\overline{M})|}{|\overline{K}|}\leqslant\frac{|\overline{M}|}{|\overline{K}|}\leqslant2nf,
\]
a contradiction. Thus $\overline{K}\geqslant\Soc(\overline{M})$. Since $\Soc(\overline{M})=\overline{M}^{(\infty)}$, this yields $(B\cap L)R\geqslant M^{(\infty)}$. As a consequence, $B\cap L\geqslant\SL_{n-1}(q)$. Note that $M^{(\infty)}\geqslant\hat{\,}\GL_{n-1}(q)\geqslant\SL_{n-1}(q)$ acts irreducibly on $q^{n-1}$. We conclude that either $B\cap L\geqslant q^{n-1}{:}\SL_{n-1}(q)$ or $B\cap L\leqslant\hat{\,}\GL_{n-1}(q)$. However, the latter causes $|A|_p|B\cap L|_p|\Out(L)|_p<|G|_p$, contrary to Lemma~\ref{Order}. Thus $B\cap L\geqslant q^{n-1}{:}\SL_{n-1}(q)$ as part~(a) asserts.
\end{proof}

Now we can state the main result of the this section.

\begin{theorem}\label{AS}
Let $L=\Soc(G)$ be a nonabelian simple group, and $G=AB$ with core-free subgroups $A,B$ of $G$. Suppose that $A$ is metacyclic. Then exactly one of the following holds.
\begin{itemize}
\item[(a)] $L=\A_n$, $\A_{n-1}\leqslant B\leqslant\Sy_{n-1}$, and $A$ is a transitive subgroup of $\Sy_n$.
\item[(b)] $L=\PSL_n(q)$ with $(n,q)\neq(2,4)$, $(2,5)$, $(2,9)$, $(3,2)$ or $(4,2)$, $A\cap L\leqslant\hat{\,}\GL_1(q^n){:}\Z_n$ and $q^{n-1}{:}\SL_{n-1}(q)\trianglelefteq B\cap L\leqslant\Pa_1$ or $\Pa_{n-1}$.
\item[(c)] $(G,A,B)$ lies in \emph{Table} $\ref{tab1}$, \emph{Table} $\ref{tab2}$ or \emph{Table} $\ref{tab4}$.
\end{itemize}
\end{theorem}

\begin{table}[htbp]
\caption{}\label{tab4}
\centering
\begin{tabular}{|l|l|l|l|}
\hline
row & $G$ & $A$ & $B$\\
\hline
1 & $\A_p$ with prime $p\equiv3\pmod{4}$ & $\Z_p{:}\Z_{(p-1)/2}$ & $\Sy_{p-2}$\\
\hline
2 & $\Sy_p$ with prime $p\equiv3\pmod{4}$ & $\Z_p{:}\Z_{(p-1)/2}$ & $\Sy_{p-2}\times\Sy_2$\\
\hline
3 & $\Sy_p$ with prime $p$ & $\Z_p{:}\Z_{p-1}$ & $\Sy_{p-2}$, $\A_{p-2}\times\Sy_2$, $\Sy_{p-2}\times\Sy_2$\\
\hline
4 & $\Sy_5$ & $\Z_6$, $\Sy_3$, $\Sy_3\times\Sy_2$ & $\AGL_1(5)$\\
\hline
5 & $\Sy_6$ & $\Z_5{:}\Z_4$ & $\Sy_3\wr\Sy_2$\\
\hline
6 & $\PGL_2(9)$ & $\Z_{10}$, $\D_{10}$, $\D_{20}$ & $\Z_3^2{:}\Z_8$\\
\hline
7 & $\PGaL_2(9)$ & $\Z_{10}{:}\Z_4$ & $\Z_3^2{:}\Z_8$, $\Z_3^2{:}\D_8$\\
\hline
\multirow{2}*{8} & \multirow{2}*{$\PGaL_2(9)$} & $\Z_{10}$, $\D_{10}$, $\D_{20}$, & \multirow{2}*
{$\Z_3^2{:}\Z_8{:}\Z_2$}\\
 & & $\Z_5{:}\Z_4$, $\Z_{10}{:}\Z_4$ &\\
\hline
\multirow{2}*{9} & \multirow{2}*{$\A_8$} & $\Z_{15}$, $\D_{10}\times\Z_3$, & \multirow{2}*
{$\AGL_3(2)$}\\
 & & $\Z_{15}{:}\Z_4$ &\\
\hline
\multirow{2}*{10} & \multirow{2}*{$\Sy_8$} & $\D_{30}$, $\Z_5\times\Sy_3$, & \multirow{2}*
{$\AGL_3(2)$}\\
 & & $\AGL_1(5)\times\Z_3$ &\\
\hline
11 & $\PSL_5(2)$ & $\Z_{31}{:}\Z_5$ & $\Z_2^6{:}(\Sy_3\times\SL_3(2))$\\
\hline
12 & $\PSL_5(2){:}\Z_2$ & $\Z_{31}{:}\Z_{10}$ & $\Z_2^6{:}(\Sy_3\times\SL_3(2))$\\
\hline
13 & $\PSU_3(8){:}\Z_3^2$ & $\Z_{57}{:}\Z_9$ & $(\Z_2^3.\Z_2^6){:}\Z_{63}{:}\Z_3$\\
\hline
14 & $\PSU_3(8){:}(\Z_3\times\Sy_3)$ & $\Z_{57}{:}\Z_{18}$ & $(\Z_2^3.\Z_2^6){:}\Z_{63}{:}\Z_3$\\
\hline
15 & $\PSU_3(8){:}(\Z_3\times\Sy_3)$ & $\Z_{57}{:}\Z_9$, $\Z_{57}{:}\Z_{18}$ & $(\Z_2^3.\Z_2^6){:}(\Z_7\times\D_{18}){:}\Z_3$\\
\hline
16 & $\PSU_4(2)$ & $\Z_9{:}\Z_3$ & $\Z_2^4{:}\A_5$\\
\hline
17 & $\PSU_4(2){:}\Z_2$ & $\Z_9{:}\Z_6$ & $\Z_2^4{:}\A_5$\\
\hline
18 & $\PSU_4(2){:}\Z_2$ & $\Z_9{:}\Z_3$, $\Z_9{:}\Z_6$ & $\Z_2^4{:}\Sy_5$\\
\hline
19 & $\M_{11}$ & $\Z_{11}{:}\Z_5$ & $\M_9{:}\Z_2$\\
\hline
20 & $\M_{11}$ & $\Z_{11}$, $\Z_{11}{:}\Z_5$ & $\M_{10}$\\
\hline
21 & $\M_{12}$ & $\Z_6\times\Z_2$, $\D_{12}$ & $\M_{11}$\\
\hline
22 & $\M_{12}{:}\Z_2$ & $\Z_{12}{:}\Z_2$ & $\M_{11}$\\
\hline
23 & $\M_{22}{:}\Z_2$ & $\D_{22}$, $\Z_{11}{:}\Z_{10}$ & $\PSL_3(4){:}\Z_2$\\
\hline
24 & $\M_{23}$ & $\Z_{23}$ & $\M_{22}$\\
\hline
25 & $\M_{23}$ & $\Z_{23}{:}\Z_{11}$ & $\M_{22}$, $\PSL_3(4){:}\Z_2$, $\Z_2^4{:}\A_7$\\
\hline
26 & $\M_{24}$ & $\D_{24}$, $\D_8\times\Z_3$ & $\M_{23}$\\
\hline
\end{tabular}
\end{table}

\begin{remark}
\emph{Two nonisomorphic groups of type $\Z_{12}{:}\Z_2$ for $A$ arise in row 22 of Table \ref{tab4}; one is isomorphic to $\D_{24}$ and the other is not.}
\end{remark}

\begin{lemma}\label{LemmaAS}
If $G$ is not a classical group of Lie type, then \emph{Theorem} $\ref{AS}$ holds.
\end{lemma}

\begin{proof}
Assume first that $\Soc(G)=\A_n$, acting naturally on $n$ points.
Then the factorization $G=AB$ is classified in \cite[Theorem D]{BOOK}, which shows that either
\begin{itemize}
\item[(i)] $\A_{n-k}\unlhd B\leqslant\Sy_{n-k}\times\Sy_k$ for some $k$ with $1\leqslant k\leqslant 5$, and $A$ is $k$-homogeneous, or

\item[(ii)] $n=6$, $8$ or $10$.
\end{itemize}
In (i), we only need to treat the case $k\geqslant2$, and note that $A,B$ may be interchanged when $n=5$.
Noticing that $A$ is metacyclic, by the classification of $k$-homogeneous groups by Kantor \cite{Kantor} (refer to \cite[p.289]{DM-book})
we conclude that either $A=\Z_p{:}\Z_{p-1}$, or $A=\Z_p{:}\Z_{(p-1)/2}$ with $p\equiv3\pmod{4}$.
For the former, $G=\Sy_p$, which leads to row 3 of Table \ref{tab4}. For the latter, either $G=\A_p$ and $B=\Sy_{p-2}$ or $G=\Sy_p$ and $B=\Sy_{p-2}\times\Sy_2$, as in row 1 and row 2 of Table \ref{tab4}, respectively. In (ii), computation by \magma~\cite{magma} shows that $\Soc(G)=\A_6\cong\PSL_2(9)$ or $\A_8$, and the
triple $(G,A,B)$ lies in rows 5--10 of Table \ref{tab4}.

If $\Soc(G)$ is a sporadic simple group, then by \cite{Giudici},
$\Soc(G)=\M_{11}$, $\M_{12}$, $\M_{22}$, $\M_{23}$ or $\M_{24}$, and the
triple $(G,A,B)$ lies in rows 19--26 of Table \ref{tab4}.

If $\Soc(G)$ is an exceptional simple group of Lie type, then by \cite{HLS},
$G$ has no metacyclic factor.
\end{proof}

%\vskip0.1in
\noindent{\bf Proof of Theorem~\ref{AS}:}

Suppose for a contradiction that there exists almost simple group $\widetilde{G}$ with socle $\widetilde{L}$ such that $\widetilde{G}=\widetilde{A}\widetilde{B}$ with $\widetilde{A}$ metacyclic and $\widetilde{B}$ core-free but $(\widetilde{G},\widetilde{A},\widetilde{B})$ is not as described in Theorem \ref{AS}. Take $\widetilde{G}$ to be such a counterexample with minimal order and let $M$ be a maximal subgroup of $\widetilde{G}$ containing $\widetilde{A}$. By Lemma \ref{LemmaAS}, $\Soc(\widetilde{G})$ is a classical group of Lie type.

If $\widetilde{A}\widetilde{L}\cap \widetilde{B}\widetilde{L}<\widetilde{G}$, then the factorization $\widetilde{A}\widetilde{L}\cap \widetilde{B}\widetilde{L}=(\widetilde{A}\cap \widetilde{B}\widetilde{L})(\widetilde{B}\cap \widetilde{A}\widetilde{L})$ (see \cite[Lemma~2(i)]{liebeck1996factorizations}) satisfies Theorem \ref{AS}. In particular, either $(\widetilde{A}\widetilde{L}\cap \widetilde{B}\widetilde{L},\widetilde{A}\cap \widetilde{B}\widetilde{L},\widetilde{B}\cap \widetilde{A}\widetilde{L})$ lies in Table \ref{tab1} or rows 11--18 of Table \ref{tab4}, or it satisfies part (b) of Theorem \ref{AS}. For the former, $\widetilde{L}$ is a linear or unitary group of prime dimension, and by Theorem \ref{2-dimension} and computation in \magma~\cite{magma}, we know that $(\widetilde{G},\widetilde{A},\widetilde{B})$ satisfies the conclusion of Theorem \ref{AS}. For the latter, since $(\widetilde{A}\cap \widetilde{B}\widetilde{L})\cap \widetilde{L}=\widetilde{A}\cap \widetilde{L}$ and $(\widetilde{B}\cap \widetilde{A}\widetilde{L})\cap \widetilde{L}=\widetilde{B}\cap \widetilde{L}$, $(\widetilde{G},\widetilde{A},\widetilde{B})$ also satisfies the conclusion of Theorem \ref{AS}. Therefore, we have $\widetilde{A}\widetilde{L}\cap \widetilde{B}\widetilde{L}=\widetilde{G}$.

As a consequence, $\widetilde{L}\nleqslant M$ for otherwise $\widetilde{A}\widetilde{L}\cap \widetilde{B}\widetilde{L}\leqslant M<G$. For the same reason, each maximal subgroup of $\widetilde{G}$ containing $\widetilde{B}$ is also core-free. Let $B$ be a maximal subgroup of $\widetilde{G}$ containing $\widetilde{B}$, and write $G=\widetilde{G}$, $L=\widetilde{L}$ and $A=\widetilde{A}$ as they will cause no confusion below. Now the factorization $G=MB$ is described in Theorem A of \cite{BOOK}, and computation by \magma~\cite{magma} shows that the $(G,A,B)$ satisfies part (b) or (c) of Theorem \ref{AS} if the triple $(G,M,B)$ is as in Table 3 of \cite{BOOK}. Consequently, $(G,M,B)$ lies in Table 1, Table 2 or Table 4 of \cite{BOOK}.

We apply the notation $\X_M$ and $\X_B$ appearing in Table 1 of \cite{BOOK} and define $\X_M=M\cap L$ and $\X_B=B\cap L$ in Table 2 and Table 4 of \cite{BOOK}. Let $R=\Rad(M)$ be the product of all solvable normal subgroups of $M$, $\overline{M}=M/R$, $Y=\X_MR/R$, $C=\C_{\overline{M}}(\Soc(Y))$, $\overline{A}=AR/R$ and $\overline{B}=(B\cap M)R/R$. Then $M=A(B\cap M)$ and so $\overline{M}/C=(\overline{A}C/C)(\overline{B}C/C)$. Moreover, $Y\cong\X_M/(\X_M\cap R)=\X_M/\Rad(\X_M)$ and $\overline{M}/C\lesssim\Aut(\Soc(Y))$.

If $\Soc(Y)$ is nonabelian simple, then $\Soc(Y)\times C\leqslant\overline{M}$ and so $\Soc(Y)\lesssim\overline{M}/C$, which implies that $\overline{M}/C$ is an almost simple group with the same socle of $Y$. Then since $|\overline{M}/C|<|G|$ and $G$ is a counterexample of minimal order, either the factorization $\overline{M}/C=(\overline{A}C/C)(\overline{B}C/C)$ satisfies Theorem \ref{AS} or $\Soc(Y)\leqslant\overline{B}C/C$. With this in mind, we analyze the factorization $G=MB$ case by case.

\textbf{Case 1.} Linear groups.

Let $L=\PSL_n(q)$, $q=p^f$ for prime number $p$ and $(n,q)\neq(3,2)$, $(4,2)$. By Theorem \ref{2-dimension}, we may assume $n\geqslant3$.

Suppose that $\X_M=\hat{\,}\GL_a(q^b).\Z_b$ with $ab=n$ and $\X_B=\Pa_1$ or $\Pa_{n-1}$. If $a=1$, then Lemma~\ref{LowerLinear} shows that $(G,A,B)$ is not a counterexample, a contradiction. Thus $a>1$ and so $A\cap L\leqslant\hat{\,}\GL_1((q^b)^a){:}\Z_a.\Z_b=\hat{\,}\GL_1(q^n){:}\Z_n$. Then again Lemma~\ref{LowerLinear} shows that $(G,A,B)$ is not a counterexample, a contradiction.

If $\X_M=\PSp_n(q)$ and $\X_B=\Pa_1$ or $\Pa_{n-1}$, then $\Soc(Y)=\PSp_n(q)$ and $\Soc(Y)\nleqslant\overline{B}C/C$. Hence $\overline{M}/C=(\overline{A}C/C)(\overline{B}C/C)$ satisfies Theorem \ref{AS}, and so we have $(n,q)=(4,3)$. However, searching in \magma~\cite{magma} for the factorizations $G=AB$ shows that this is not possible. Checking other candidates in Table 1 of \cite{BOOK} similarly, we obtain Table \ref{tab5}.
\begin{table}[!h]
\caption{}\label{tab5}
\centering
\begin{tabular}{|l|l|l|l|}
\hline
row & $\X_M$ & $\X_B$ & Conditions\\
\hline
1 & $\Pa_1$, $\Pa_{n-1}$ & $\hat{\,}\GL_a(q^b).\Z_b$ & $n=ab$, $b$ prime\\
\hline
2 & $\mathrm{Stab}(V_1\oplus V_{n-1})$ & $\hat{\,}\GL_a(q^2).\Z_2$ & $n=2a$, $q=2$ or $4$\\
\hline
3 & $\Pa_1$, $\Pa_{n-1}$ & $\PSp_n(q)$ & $n$ even\\
\hline
4 & $\mathrm{Stab}(V_1\oplus V_{n-1})$ & $\PSp_n(q)$ & $n$ even\\
\hline
\end{tabular}
\end{table}

For rows 3 and 4 of Table \ref{tab5}, either $p^{f(n-1)}-1$ has a primitive prime divisor $r$ or $(n,q)=(4,4)$. If $(n,q)\neq(4,4)$, then since $\Z_r$ has no faithful representation over $\GF(p)$ of dimension less than $(n-1)f$, the congruence $|A|\equiv0\pmod{pr}$ implies $\Z_p^{f(n-1)}\leqslant A$, contradicting that $A$ is metacyclic. If $(n,q)=(4,4)$, then since $|A\cap L||\Out(L)|$ is divisible by $|L|/|B\cap L|$, $|A\cap L|$ is divisible by $252$, which is impossible as $M\cap L$ has no metacyclic subgroup of order divisible by $252$.

For rows 1 and 2 of Table \ref{tab5}, since $|B\cap L|$ divides $|\GL_a(q^b)|\cdot b$, we deduce from Lemma \ref{Order} that $b|A\cap L||\Out(L)|$ is divisible by $q^{n(n-1)/2-ab(a-1)/2}=q^{n(n-a)/2}$. In light of Lemma \ref{Exponent}, the Sylow $p$-subgroup of $A\cap L$ has order dividing $p^{2\lceil\log_pn\rceil}$. Hence we have
\begin{equation}\label{eq5}
q^{n(n-a)/2}\di2fbp^{2\lceil\log_pn\rceil}.
\end{equation}

(i). $n=3$. Then $(a,b)=(1,3)$, and (\ref{eq5}) turns out to be $q^3\di6fp^{2\lceil\log_p3\rceil}$. If $p=2$, then $f=2$ as $(n,q)\neq(3,2)$. If $p\geqslant3$, then $\lceil\log_p3\rceil=1$ and so $(p,f)=(3,1)$. However, computation in \magma~\cite{magma} shows that there is no counterexample of Theorem \ref{AS} when $L=\PSL_3(3)$ or $\PSL_3(3)$.

(ii). $n=4$. Then $(a,b)=(2,2)$, and (\ref{eq5}) turns out to be $q^4\di4fp^{2\lceil\log_p4\rceil}$. If $p\leqslant3$, then $(p,f)=(3,1)$ as $(n,q)\neq(4,2)$. If $p\geqslant5$, then $\lceil\log_p3\rceil=1$ and it is impossible that $q^4\di4fp^2$. Accordingly, $q=3$. However, computation by \magma~\cite{magma} shows that there is no counterexample of Theorem \ref{AS} when $L=\PSL_4(3)$.

(iii). $n=5$. Then $(a,b)=(1,5)$, and (\ref{eq5}) turns out to be $q^{10}\di10fp^{2\lceil\log_p5\rceil}$. Since $p^{2\lceil\log_pn\rceil}<p^{2(\log_pn+1)}=n^2p^2$, we derive $q^{10}<250fp^2$, a contradiction.

(iv). $n\geqslant6$. Since $a\leqslant n/2$, $b\leqslant n$ and $p^{2\lceil\log_pn\rceil}<n^2p^2$, (\ref{eq5}) implies $q^{n^2/4}<2fn^3p^2$, that is, $p^{fn^2/4-2}<2fn^3$. Consequently, $2^{fn^2/4-2}<2fn^3$, which indicates $n=6$ and $f=1$. Then there is no solution for (\ref{eq5}), a contradiction.

\textbf{Case 2.} Symplectic groups.

Let $L=\PSp_{2m}(q)$, $m\geqslant2$, $q=p^f$ for prime number $p$ and $(m,q)\neq(2,2)$. For $(m,q)=(2,3)$, $(2,4)$, $(3,2)$, $(4,2)$ computation in \magma~\cite{magma} directly shows that the factorization $G=AB$ satisfies Theorem~\ref{AS}. Thus we have $(m,q)\neq(2,3)$, $(2,4)$, $(3,2)$ or $(4,2)$.

If $\X_M=\PSp_{2a}(q^b).\Z_b$ with $ab=m$ and $b$ prime and $\X_B=\Pa_1$, then $\Soc(Y)=\PSp_{2a}(q^b)$ and $\Soc(Y)\nleqslant\overline{B}C/C$. This implies that $\overline{M}/C=(\overline{A}C/C)(\overline{B}C/C)$ satisfies Theorem \ref{AS}, and so row 1 of Table \ref{tab6} occurs. Checking other candidates in Table 1 and Table 2 of \cite{BOOK} similarly, we obtain rows 2--10 of table \ref{tab6}.
\begin{table}[!h]
\caption{}\label{tab6}
\centering
\begin{tabular}{|l|l|l|l|}
\hline
row & $\X_M$ & $\X_B$ & Conditions\\
\hline
1 & $\PSp_2(q^m).\Z_m$ & $\Pa_1$ & $m$ prime\\
\hline
2 & $\PSp_2(q^m).\Z_m$ & $\PSO^\pm_{2m}(q)$ & $m$ prime, $p=2$\\
\hline
3 & $\PSO^\pm_4(q)$ & $\PSp_2(q^2).\Z_2$ & $m=2$, $p=2$\\
\hline
4 & $\PSO^+_6(q)$ & $\PSp_2(q^3).\Z_3$ & $m=3$, $p=2$\\
\hline
5 & $\PSO^-_4(q)$ & $\Pa_2$ & $m=2$, $p=2$\\
\hline
6 & $\Pa_m$ & $\PSO^-_{2m}(q)$ & $p=2$\\
\hline
7 & $\PSO^-_4(q)$ & $\PSp_2(q)\wr\Sy_2$ & $m=2$, $p=2$\\
\hline
8 & $\PSp_m(q)\wr\Sy_2$ & $\PSO^-_{2m}(q)$ & $m$ even, $p=2$\\
\hline
9 & $\PSO^+_4(q)$ & $\Sz(q)$ & $m=2$, $p=2$, $f$ odd\\
\hline
10 & $\PSO^+_6(q)$ & $\G_2(q)$ & $m=3$, $p=2$\\
\hline
\end{tabular}
\end{table}

For row 1 or 5 of Table \ref{tab6}, $|A\cap L||\Out(L)|$ is divisible by $(q^{2m}-1)/(q-1)$, and hence $\PSL_2(q^m)$ has a metacyclic subgroup $Y$ such that $2fm|Y|(q-1)$ is divisible by $q^{2m}-1$, not possible. For row 2 of Table \ref{tab6}, $|L|/|B\cap L|$ is divisible by $q^m/2$, and so is $|A\cap L||\Out(L)|$. Note that every metacyclic $2$-subgroup of $M\cap L$ has order dividing $4m$. We then derive a contradiction that $2^{fm}\di16mf$, which excludes row 2 of Table \ref{tab6}. Similar argument excludes rows 3, 4, 7 and 9 of Table \ref{tab6}, as well as row 10 for $q\neq4$. If $q=4$ in row 10 of Table \ref{tab6}, then $M\cap L$ must contain a cyclic subgroup of order $2^2\cdot17$ since $|A\cap L|$ is divisible by $2^5\cdot17$, but $\PSO^+6(4)$ has no element of order $68$, a contradiction. Now either row 6 or row 8 of Table \ref{tab6} occurs. Then $|A\cap L||\Out(L)|$ is divisible by $q^m/2$, and so $2^{1+2(1+\lceil\log_2m\rceil)}|\Out(L)|$ is divisible by $q^m$, from which we deduce that $(m,f)=(2,3)$, $(2,4)$, $(3,2)$, $(4,2)$, $(5,1)$, $(5,2)$, $(6,1)$, $(7,1)$, $(8,1)$, $(9,1)$, $(10,1)$ or $(11,1)$.

(i). $m=2$ and $f=3$. Since $|G|/|B|=|L|/|B\cap L|$ is divisible by $2^5\cdot7$, $A\cap L$ is divisible by $2^4\cdot7$ and hence contains a cyclic subgroup of order $2^2\cdot7$. However, $\PSp_4(2^3)$ has no element of order $28$, a contradiction.

(ii). $m=2$ and $f=4$. Since $|G|/|B|$ is divisible by $2^7\cdot17$, $A\cap L$ is divisible by $2^4\cdot17$ and hence contains a cyclic subgroup of order $2^2\cdot17$. However, $\PSp_4(2^4)$ has no element of order $68$, a contradiction.

(iii). $m=3$, $f=2$. Then $M\cap L=\Z_2^{12}{:}\GL_3(4)$ and $A\cap L$ is divisible by $2^5\cdot7$. This indicates that there is an element of order $14$ in $\GL_3(4)$, not possible.

(iv). $m=4$, $f=2$. Then $M\cap L=\Z_2^{20}{:}\GL_4(4)$ and $A\cap L$ is divisible by $2^6\cdot17$. This shows that there is an element of order $68$ in $\GL_4(4)$, not possible.

(v). $m=5$, $f=1$. Then $M\cap L=\Z_2^{15}{:}\GL_5(2)$ and $A\cap L$ is divisible by $2^4\cdot31$. This implies that there is an element of order $62$ in $\GL_5(2)$, not possible.

(vi). $m=5$, $f=2$. Then $M\cap L=\Z_2^{30}{:}\GL_5(4)$ and $A\cap L$ is divisible by $2^8\cdot11$. This implies that there is an element of order $88$ in $\GL_5(4)$, not possible.

(vii). $m=6$, $f=1$. If $M=\Z_2^{21}{:}\GL_6(2)$, then $\overline{M}=\overline{A}~\overline{B}$ implies $\overline{A}\leqslant\GaL_3(4)$ or $\GaL_2(8)$, but neither $\GaL_3(4)$ nor $\GaL_2(8)$ has a metacyclic subgroup of order divisible by $2^3\cdot63$, contradicting the condition that $|A|$ is divisible by $2^5\cdot63$. If $M=\PSp_6(2)\wr\Sy_2$, then $B\cap M$ is contained in the subgroup $N=\PSp_6(2)\times\PSp_6(2)$ of $M$ since $|B\cap M|=|M||B|/|G|$, so $N=(A\cap N)(B\cap M)$, which yields the contradiction that $\PSp_6(2)$ has a factorization with a metacyclic factor.

(viii). $7\leqslant m\leqslant11$, $f=1$. This is also not possible for the similar reason as (vii).

\textbf{Case 3.} Unitary groups.

Note that there are no factorizations for unitary groups of odd dimension in Table 1 and Table 2 of \cite{BOOK}. Let $L=\PSU_{2m}(q)$, $m\geqslant2$ and $q=p^f$ for prime number $p$. For $(m,q)=(2,2)$, $(2,3)$, $(3,2)$ computation in \magma~\cite{magma} directly shows that the factorization $G=AB$ satisfies Theorem~\ref{AS}. Thus we have $(m,q)\neq(2,2)$, $(2,3)$, $(3,2)$.

If $\X_M=\N_1$ and $\X_B=\Pa_m$, then $\Soc(Y)=\PSU_{2m-1}(q)$ and $\Soc(Y)\nleqslant\overline{B}C/C$. Hence $\overline{M}/C=(\overline{A}C/C)(\overline{B}C/C)$ satisfies Theorem \ref{AS}, and so we have $(m,q)=(2,8)$. However, computation by \magma~\cite{magma} shows that this does not give rise to the factorization $G=AB$ as required. Checking other candidates in Table 1 of \cite{BOOK} similarly, we obtain table \ref{tab7}.
\begin{table}[!h]
\caption{}\label{tab7}
\centering
\begin{tabular}{|l|l|l|l|}
\hline
row & $\X_M$ & $\X_B$ & Conditions\\
\hline
1 & $\Pa_m$ & $\N_1$ & \\
\hline
2 & $\hat{\,}\SL(m,4).\Z_2$ & $\N_1$ & $m\geqslant3$, $q=2$\\
\hline
3 & $\hat{\,}\SL(m,16).\Z_3.\Z_2$ & $\N_1$ & $q=4$\\
\hline
\end{tabular}
\end{table}

For rows 2 and 3 of Table \ref{tab7}, $|L|/|B\cap L|$ is divisible by $q^{2m-1}$, and so $M$ must have a metacyclic subgroup of order divisible by $q^{2m-1}$. However, this implies $2f\cdot2^{2\lceil\log_2m\rceil+1}\geqslant2^{f(2m-1)}$, which is not possible. Now row 1 of Table \ref{tab7} occurs. Then $2fp^{2+2\lceil\log_pm\rceil}$ is divisible by $q^{2m-1}$. It follows that either $m=2$ and $q=p\geqslant5$, or $(m,q)=(2,4)$, $(4,2)$ or $(5,2)$.

(i). $m=2$, $q=4$. Then $A\cap L$ is divisible by $2^4\cdot17$, which indicates that $\PSL_2(16)$ has a metacyclic subgroup of order divisible by $68$, not possible.

(ii). $m=2$, $q=p\geqslant5$. Let $r$ be a primitive prime divisor of $p^4-1$.
Then $|A\cap L|$ is divisible by $pr$, and so $\PSL_2(p^2)$ has a metacyclic subgroup of order divisible by $pr$, not possible.

(iii). $m=4$, $q=2$. Here $A\cap L$ is divisible by $2^6\cdot17$. Then considering the factorization $\overline{M}/C=(\overline{A}C/C)(\overline{B}C/C)$ we conclude that $\GL_2(16)$ has a metacyclic subgroup of order divisible by $68$, which is not possible.

(iv). $m=5$, $q=2$. Here $A\cap L$ is divisible by $2^8$. Then considering the factorization $\overline{M}/C=(\overline{A}C/C)(\overline{B}C/C)$, we conclude that $\GL_1(4^5)$ has a metacyclic subgroup of order divisible by $2$, not possible.

\textbf{Case 4.} Orthogonal groups of odd dimension.

Let $L=\POm_{2m+1}(q)$, $m\geqslant3$ and $q=p^f$ for prime number $p\geqslant3$. For $(m,q)=(3,3)$ computation in \magma~\cite{magma} directly shows that the factorization $G=AB$ satisfies Theorem~\ref{AS}. Thus we have $(m,q)\neq(3,3)$.

If $\X_M=\N_1^-$ and $\X_B=\Pa_m$, then $\Soc(Y)=\POm^-_{2m}(q)$ and $\Soc(Y)\nleqslant\overline{B}C/C$. This implies that $\overline{M}/C=(\overline{A}C/C)(\overline{B}C/C)$ satisfies Theorem \ref{AS}, a contradiction. Checking other candidates in Table 1 and Table 2 of \cite{BOOK} similarly, we conclude that either $M\cap L=\Pa_m$ and $B\cap L=\N_1^-$, or $M\cap L=\N_1^+$ and $B\cap L=\G_2(q)$ with $m=3$. The latter is impossible because $\PSL_4(q)$ does not contain a solvable subgroup of order divisible by $pr$, where $r$ is a primitive prime divisor of $p^{4f}-1$. For the former, $p^{fm}\di fp^{2\lceil\log_p(2m+1)\rceil}$ implies $(m,q)=(3,5)$, $(4,3)$, $(4,5)$, $(5,3)$ or $(6,3)$.

(i). $m=3$, $q=5$. Since $|A\cap L|$ is divisible by $5^3\cdot31$, the factor group $\Z_5^3{:}\GL_3(5)$ of $M\cap L$ modulo $\Z_5^3$ must contain a solvable subgroup of order divisible by $155$, which is not possible.

(ii). $m=4$, $q=3$ or $5$. Let $r$ be a primitive prime divisor of $p^4-1$. Then it is a contradiction that the factor group $\Z_p^4{:}\GL_4(p)$ of $M\cap L$ modulo $\Z_p^6$ contains a solvable subgroup of order divisible by $p^2r$.

(iii). $m=5$, $q=3$. Then $|A\cap L|$ is divisible by $3^5\cdot11$, which indicates that $\PSL_5(3)$ has a metacyclic subgroup of order divisible by $33$, not possible.

(iv). $m=6$, $q=3$. Here $|A\cap L|$ is divisible by $3^6\cdot7$. Then considering the factorization $\overline{M}/C=(\overline{A}C/C)(\overline{B}C/C)$ we conclude that either $\PGaL_2(3^3)$ or $\PGaL_3(3^2)$ has a solvable subgroup of order divisible by $63$, which is not true.

\textbf{Case 5.} Orthogonal groups of even dimension.

Let $L=\POm^\varepsilon_{2m}(q)$ with $\varepsilon=\pm$, $m\geqslant4$, $q=p^f$ for prime number $p$ and $d=\gcd(2,q-1)$. For $(m,q)=(4,2)$, $(4,3)$, $(5,2)$, $(6,2)$ computation in \magma~\cite{magma} directly shows that the factorization $G=AB$ satisfies Theorem~\ref{AS}. Thus we have $(m,q)\neq(4,2)$, $(4,3)$, $(5,2)$ or $(6,2)$.

If $\X_M=\hat{\,}\GU(m,q)$ and $\X_B=\Pa_1$ with odd $m\geqslant5$, then $\Soc(Y)=\PSU(m,q)$ and $\Soc(Y)\nleqslant\overline{B}C/C$. This implies that $\overline{M}/C=(\overline{A}C/C)(\overline{B}C/C)$ satisfies Theorem \ref{AS}, a contradiction. Checking other candidates in Table 1, Table 2 and Table 4 of \cite{BOOK} similarly, we obtain table \ref{tab8}.
\begin{table}[!h]
\caption{}\label{tab8}
\centering
\begin{tabular}{|l|l|l|l|}
\hline
row & $\X_M$ & $\X_B$ & Conditions\\
\hline
1 & $\Pa_1$ & $\hat{\,}\GU_m(q)$ & $\varepsilon=-$, $m\geqslant5$, $m$ odd\\
\hline
2 & $\Pa_1$, $\Pa_3$ or $\Pa_4$ & $\Omega_7(q)$ & $\varepsilon=+$, $m=4$\\
\hline
3 & $\hat{\,}((q+1)/d\times\Omega^+_6(q)).2^d$ & $\Omega_7(q)$ & $\varepsilon=+$, $m=4$, $q>2$\\
\hline
4 & $\Pa_1$, $\Pa_3$ or $\Pa_4$ & $\hat{\,}((q+1)/d\times\Omega^-_6(q)).2^d$ & $\varepsilon=+$, $m=4$\\
\hline
5 & $\Pa_m$ or $\Pa_{m-1}$ & $\N_1$ & $\varepsilon=+$, $m\geqslant5$\\
\hline
6 & $\hat{\,}\GL_m(q).\Z_2$ & $\N_1$ & $\varepsilon=+$, $m\geqslant5$\\
\hline
7 & $\Pa_m$ or $\Pa_{m-1}$ & $\N_2^-$ & $\varepsilon=+$, $m\geqslant5$\\
\hline
\end{tabular}
\end{table}

If row 1 of Table \ref{tab8} occurs, then $2fp^{2+2\lceil\log_p(m-2)\rceil}\geqslant p^{fm(m-1)/2}$, which is not possible. If row 4 of Table \ref{tab8} occurs, then $p^{6f}\di12fp^{2\lceil\log_p8\rceil}$, still not possible. For rows 2 and 3 of Table \ref{tab8}, $p^{3f}\di6fp^{2\lceil\log_p8\rceil}$ implies $q=4,5$ or $7$. These are impossible since $\PSL_4(q)$ does not contain a solvable subgroup of order divisible by $pr$, where $r$ is a primitive prime divisor of $p^{4f}-1$. If row 7 of Table \ref{tab8} occurs, then $p^{2f(m-1)}\di4fp^{2\lceil\log_p2m\rceil}$, not possible. Now row 5 or 6 of Table \ref{tab8} occurs. We then have $p^{f(m-1)}\di2fp^{2\lceil\log_p2m\rceil}$, which implies that $(m,q)=(5,3)$, $(5,4)$, $(5,5)$, $(5,7)$, $(6,3)$, $(6,4)$, $(7,2)$, $(7,3)$, $(8,2)$, $(9,2)$, $(10,2)$, $(11,2)$ or $(12,2)$.

(i). $m=5$, $q=3,4,5 $ of $7$. In view of the factorization $\overline{M}/C=(\overline{A}C/C)(\overline{B}C/C)$ and the condition that $|A|$ is divisible by $|L|/|B\cap L|$, we conclude that $\GL_1(q^5){:}\Z_5$ contains a solvable subgroup of order divisible by $q^2$, which is not true.

(ii). $m=6$, $q=3$ or $4$. In view of the factorization $\overline{M}/C=(\overline{A}C/C)(\overline{B}C/C)$ and the condition that $|A|$ is divisible by $|L|/|B\cap L|$, we conclude that either $\PSL_2(q^3){:}\Z_3$ or $\PSL_3(q^2){:}\Z_2$ contains a metacyclic subgroup of order divisible by $q^3$, not possible.

(iii). $m=7$, $q=2$ or $3$. In view of the factorization $\overline{M}/C=(\overline{A}C/C)(\overline{B}C/C)$ and the condition that $|A|$ is divisible by $|L|/|B\cap L|$, we conclude that $\GL_1(q^7){:}\Z_7$ contains a metacyclic subgroup of order divisible by $q^3$, impossible.

(iv). $m=8$, $q=2$. In view of the factorization $\overline{M}/C=(\overline{A}C/C)(\overline{B}C/C)$ and the condition that $|A|$ is divisible by $|L|/|B\cap L|$, we conclude that $\PGaL_4(4)$ contains a metacyclic subgroup of order divisible by $2^5\cdot17$, not possible.

(v). $9\leqslant m\leqslant12$, $q=2$. Then the same argument as above shows that this is not possible either. The proof is thus completed.
\qed

\section{Affine type}

In this section, we determine quasiprimitive permutation groups of affine type with a metacyclic transitive subgroup.
Notice that quasiprimitive permutation groups of affine type are primitive.

\begin{theorem}\label{HA}
Let $G\leqslant\Sym(\Ome)$, $\omega\in\Ome$, and $R$ be a metacyclic transitive subgroup of $G$. If $G$ is primitive with socle $\Z_p^n$ for some prime number $p$ and positive integer $n$, then one of the following holds.
\begin{itemize}
\item[(a)] $n=1$, and $\Z_p\leqslant R\leqslant G\leqslant\Z_p{:}\Z_{p-1}$.
\item[(b)] $n=2$, $\Z_p.\Z_p\leqslant R\leqslant G\leqslant\AGL_2(p)$, and $G_\omega$ is an irreducible subgroup of $\GL_2(p)$.
\item[(c)] $n=3$, $p=2$, $G=\AGL_3(2)$, and $R=\Z_4\times\Z_2$, $\D_8$ or $\Q_8$.
\item[(d)] $n=3$, $p=3$, $G\leqslant\AGL_3(3)$, $R=\Z_9\times\Z_3$, $3_-^{1+2}$ or $\Z_9{:}\Z_6$, and $G_\omega=\A_4$, $\Sy_4$, $\A_4\times\Z_2$, $\Sy_4\times\Z_2$, $\Z_{13}{:}\Z_3$, $\Z_{13}{:}\Z_6$, $\SL_3(3)$ or $\GL_3(3)$.
\item[(e)] $n=4$, $p=2$, $G\leqslant\AGL_4(2)$, $R=\Z_4{:}\Z_4$, $\Z_8{:}\Z_2$ or $\Q_{16}$, and $G_\omega=\AGL_1(5)$, $\Z_3^3{:}\Z_4$, $\Z_{15}{:}\Z_4$, $\Z_3^2{:}\D_8$, $\A_5$, $\Sy_5$, $\A_5\times\Z_3$, $\A_5{:}\Sy_3$, $\A_6$, $\Sy_6$, $\A_7$ or $\GL_4(2)$.
\end{itemize}
\end{theorem}

\begin{proof}
As $\Soc(G)=\Z_p^n$, we can view $G$ as a subgroup of $\AGL_n(p)$ acting on the $n$-dimensional vector space over $\GF(p)$. The transitivity of $R$ implies that $|R|$ is divisible by $p^n$. Let $P$ be a Sylow $p$-subgroup of $R$. Then $P$ is metacyclic, and $|P|$ is divisible by $p^n$. Consequently, there exists an element of order divisible by $p^{\lceil n/2\rceil}$ in $R$. Note that $P\leqslant G\leqslant\AGL_n(p)\leqslant\GL_{n+1}(p)$. We then conclude from Lemma \ref{Exponent} that any element of $P$ has order dividing $p^{\lceil\log_p(n+1)\rceil}$, and so
\begin{equation}\label{eq1}
\lceil n/2\rceil\leqslant\lceil\log_p(n+1)\rceil.
\end{equation}
Thus, $n/2\leqslant\lceil n/2\rceil<\log_p(n+1)+1\leqslant\log_2(n+1)+1$, which implies $n\leqslant8$. Substituting $n=8$ into (\ref{eq1}) we obtain $4\leqslant\lceil\log_p9\rceil$, and hence $p=2$. Similar argument for $n=3,4,5,6,7$ gives all the solutions for (\ref{eq1}) when $n\geqslant3$:
\begin{equation}\label{eq2}
(n,p)\in\{(3,2),(3,3),(4,2),(4,3),(5,2),(6,2),(8,2)\}.
\end{equation}

The primitivity of $G$ implies that $G_\omega$ is an irreducible subgroup of $\GL_n(p)$. If $n=1$ or $2$, then since $|R|$ is divisible by $p^n$, we obtain part~(a) or~(b). Next assume $n\geqslant3$. To complete the proof, we only need to discuss the candidates for $(n,p)$ in (\ref{eq2}).

\textbf{Case 1.} $n=3$ and $p=2$. Checking the database of affine primitive groups of degree $8$ in \magma~\cite{magma}, we know that $G=\AGL_3(2)$, and $R=\Z_4\times\Z_2$, $\D_8$ or $\Q_8$ as $R$ is a metacyclic transitive subgroup of $G$.

\textbf{Case 2.} $n=3$ and $p=3$. In this case, $G\leqslant\AGL_3(3)$. Searching the metacyclic transitive subgroups of $\AGL_3(3)$ and checking the database of affine primitive groups of degree $27$ in \magma~\cite{magma}, we obtain the possibilities for $R$ and $G_\omega$ as described in part~(d).

\textbf{Case 3.} $n=4$ and $p=2$. Then $G\leqslant\AGL_4(2)$, and searching in \magma~\cite{magma} as above produces the possibilities for $R$ and $G_\omega$ as in part~(e).

\textbf{Case 4.} $n=4$ and $p=3$. Since $R<\AGL_4(3)\leqslant\GL_5(3)$, we know from Lemma \ref{Exponent} that the order of a $3$-element in $R$ is at most $9$. However, $3^4$ divides $|R|$ since $R$ is a transitive subgroup of $\AGL_4(3)$. Therefore, $R$ is of form $\Z_9{:}\Z_9$. However, computation in \magma~\cite{magma} shows that $\AGL_4(3)$ has no transitive subgroup isomorphic to $\Z_9{:}\Z_9$, a contradiction.

\textbf{Case 5.} $n=5$ and $p=2$. Since $R<\AGL_5(2)\leqslant\GL_6(2)$, we know from Lemma \ref{Exponent} that the order of a $2$-element in $R$ is at most $8$. However, $2^5$ divides $|R|$ since $R$ is a transitive subgroup of $\AGL_5(2)$. Therefore, $R$ is of form $\Z_4{:}\Z_8$, $\Z_8{:}\Z_4$ or $\Z_8{:}\Z_8$. However, computation in \magma~\cite{magma} shows that no transitive subgroup of $\AGL_5(2)$ has one of these forms, a contradiction.

\textbf{Case 6.} $n=6$ or $8$, and $p=2$. In this case, $R$ is a metacyclic subgroup of $\AGL_n(2)$ with order divisible by $2^n$. If $n=6$, then $\AGL_6(2)\leqslant\GL_7(2)<\GL_8(2)$ implies that $\GL_8(2)$ has a metacyclic subgroup of order divisible by $2^6$. If $n=8$, then $R\cap\Soc(G)\leqslant\Z_2^2$ implies that the subgroup $R\Soc(G)/\Soc(G)\cong R/R\cap\Soc(G)$ of $\GL_8(2)$ has order divisible by $2^6$. Thus we always have a subgroup $\overline{R}$ of $\GL_8(2)$ whose order is divisible by $2^6$. Since the largest order of a $2$-element in $\GL_8(2)$ equals $8$ by Lemma \ref{Exponent}, we deduce that $\overline{R}=\Z_8{:}\Z_8$. However, computation in \magma~\cite{magma} shows that $\GL_8(2)$ has no subgroup of form $\Z_8{:}\Z_8$, a contradiction.
\end{proof}

\section{Diagonal type}

In this section we determine quasiprimitive permutation groups of diagonal type with a metacyclic transitive subgroup. As we mentioned in the introduction section, quasiprimitive permutation groups of diagonal type include types holomorph simple and simple diagonal.

Recall that for a group $G$, the \emph{holomorph} of $G$, denoted by $\Hol(G)$, is the normalizer of the right regular representation of $G$ in $\Sym(G)$, and has the structure $G{:}\Aut(G)$.

\begin{lemma}\label{NormalSemireg}
Let $G\leqslant\Sym(\Ome)$, and $R$ be a metacyclic transitive subgroup of $G$. If $G$ has a semiregular normal subgroup $N\cong T^k$ with nonabelian simple group $T$ and positive integer $k$, then $k=1$ and $T=\PSL_2(p)$ for some prime $p\geqslant5$.
\end{lemma}

\begin{proof}
Take an orbit $\Del$ of $N$ on $\Ome$. Then $N^\Del$ is a regular subgroup of $\Sym(\Del)$ and $N^\Del\cong N$. Since $R$ is a metacyclic transitive group on $\Ome$, the induced permutation group $\overline{R}=R_\Del^\Del$ is also metacyclic and transitive on $\Del$.

Let $M=N^\Del\overline{R}\leqslant\Sym(\Del)$. Note that $N^\Del$ is a regular normal subgroup of $M$. We have $M\leqslant\Hol(N^\Del)\cong\Hol(T)\wr\Sy_k$ and that
\begin{equation}\label{eq3}
|\overline{R}|\text{ is divisible by }|T|^k.
\end{equation}
Moreover, for a point $\omega$ in $\Del$, we have $M=\overline{R}M_\omega$ and $M=N^\Del{:}M_\omega$ as $\overline{R}$ and $N^\Del$ are both transitive on $\Del$. It follows that $M_\omega\cong M/N^\Del=N^\Del\overline{R}/N^\Del\cong\overline{R}/\overline{R}\cap N^\Del$
is metacyclic. Now $M=\overline{R}M_\omega$ is a product of two metacyclic groups, so by \cite{Kazarin}, $T=\PSU_3(8)$, $\PSp_4(3)$, $\PSL_4(2)$, $\M_{11}$, $\PSL_3(q)$ with $q\in\{3,4,5,7,8\}$ or $\PSL_2(q)$ with prime power $q\geqslant5$.

\textbf{Case 1.} $T=\PSU_3(8)$. Note that the exponent of the Sylow $2$-subgroup of $\PSU_3(8)$ equals $4$, and the Sylow $2$-subgroup of $\Out(\PSU_3(8))$ is $\Z_2$. We see that the exponent of the Sylow $2$-subgroup of $\Hol(T)=\PSU_3(8)^2.\Out(\PSU_3(8))$ divides $4\cdot2=8$, and so does that of $\Hol(T)^k$. Observe that a $2$-element of $\Sy_k$ has order at most $k$. We then conclude that a $2$-element of $\Hol(T)\wr\Sy_k=\Hol(T)^k.\Sy_k$ has order at most $8k$, and hence the Sylow $2$-subgroup of $\overline{R}$ has order at most $(8k)^2$. However, the Sylow $2$-subgroup of $\overline{R}$ has order divisible by $|\PSU_3(8)|_2^k=2^{9k}$ due to (\ref{eq3}). Therefore we obtain that $2^{9k}\leqslant(8k)^2$, which is not possible.

\textbf{Case 2.} $T=\PSp_4(3)$. Note that the exponent of the Sylow $2$-subgroup of $\PSp_4(3)$ equals $4$, and $\Out(\PSp_4(3))=\Z_2$. Similarly as the previous case we see that the exponent of the Sylow $2$-subgroup of $\Hol(T)=\PSp_4(3)^2.\Out(\PSp_4(3))$ divides $4\cdot2=8$, and thus the Sylow $2$-subgroup of $\overline{R}$ has order at most $(8k)^2$. By (\ref{eq3}), the Sylow $2$-subgroup of $\overline{R}$ has order divisible by $|\PSp_4(3)|_2^k=2^{6k}$. Then we obtain the inequality $2^{6k}\leqslant(8k)^2$, which forces $k=1$. Computation in \magma~\cite{magma} shows that the maximal order of metacyclic subgroups in $\PSp_4(3)$ is $36$. Hence we deduce that $|\overline{R}|\leqslant36^2|\Out(\PSp_4(3))|=2592$. This violates (\ref{eq3}) as $|T|=|\PSp_4(3)|=25920$.

\textbf{Case 3.} $T=\PSL_4(2)$. Note that the exponent of the Sylow $2$-subgroup of $\PSL_4(2)$ equals $4$, and $\Out(\PSL_4(2))=\Z_2$. We conclude as above that the exponent of the Sylow $2$-subgroup of $\Hol(T)=\PSL_4(2)^2.\Out(\PSL_4(2))$ divides $8$, and thus the Sylow $2$-subgroup of $\overline{R}$ has order at most $(8k)^2$. By (\ref{eq3}), the Sylow $2$-subgroup of $\overline{R}$ has order divisible by $|\PSL_4(2)|_2^k=2^{6k}$. Then we obtain the inequality $2^{6k}\leqslant(8k)^2$, which forces $k=1$. Computation in \magma~\cite{magma} shows that the maximal order of metacyclic subgroups in $\PSL_4(2)$ is $60$. Hence $|\overline{R}|\leqslant60^2|\Out(\PSL_4(2))|=7200$, contradicting (\ref{eq3}) as $|T|=20160$.

\textbf{Case 4.} $T=\M_{11}$. In this case, $\Hol(T)\cong T^2$ since $\Out(T)=1$. Note that the exponent of the Sylow $3$-subgroup of $\M_{11}$ is $3$. We conclude from $\overline{R}\leqslant M\lesssim\Hol(T)\wr\Sy_k=\Hol(T)^k.\Sy_k$ that the Sylow $3$-subgroup of $\overline{R}$ has order at most $(3k)^2$, and thus $3^{2k}=|\M_{11}|_3^k\leqslant(3k)^2$ by (\ref{eq3}). Consequently, $k=1$ and $\overline{R}\leqslant \Hol(T)\cong T^2$. Computation in \magma~\cite{magma} shows that the maximal order of metacyclic subgroups in $\M_{11}$ is $55$. Hence $|\overline{R}|\leqslant55^2=3025$, contradicting (\ref{eq3}) as $|T|=7920$.

\textbf{Case 5.} $T=\PSL_3(q)$ with $q\in\{3,4,5,7,8\}$. Let $q=p^f$ with $p$ prime. First assume $q\in\{3,5,7\}$. As $p=q\geqslant3$, we deduce from Lemma \ref{Exponent} that the Sylow $p$-subgroup of $T$ has exponent $p$. Then it follows from $|\Out(T)|_p=1$ that the exponent of the Sylow $p$-subgroup of $\Hol(T)$ is $p$. Since $\overline{R}\lesssim\Hol(T)\wr\Sy_k=\Hol(T)^k.\Sy_k$, we conclude that the Sylow $p$-subgroup of $\overline{R}$ has order at most $(pk)^2$. However, (\ref{eq3}) shows that the Sylow $p$-subgroup of $\overline{R}$ has order divisible by $|T|_p^k=p^{3k}$. Thereby we have $p^{3k}\leqslant(pk)^2$, which yields $3^{3k-2}\leqslant p^{3k-2}\leqslant k^2$, a contradiction.

Next assume $q=2^f$ with $f=2$ or $3$. Then the Sylow $2$-subgroup of $T$ has exponent $4$ by Lemma \ref{Exponent}, and the exponent of the Sylow $2$-subgroup of $|\Out(T)|$ is $2$. Hence every $2$-element in $\Hol(T)$ has order dividing $8$, and so the order of the Sylow $2$-subgroup of $\overline{R}$ is at most $(8k)^2$. Now we deduce from (\ref{eq3}) that $2^{3fk}=|T|_2^k\leqslant(8k)^2$. This forces $f=2$ and $k=1$, whence $\overline{R}\leqslant\PSL_3(4)^2.\Out(\PSL_3(4))$. However, the maximal order of metacyclic subgroups in $\PSL_3(4)$ is $21$ according to computation in \magma~\cite{magma}. Thus $|\overline{R}|\leqslant21^2|\Out(\PSL_3(4))|=5292$, contradicting (\ref{eq3}) as $|T|=|\PSL_3(4)|=20160$.

\textbf{Case 6.} $T=\PSL_2(q)$ with $q=p^f\geqslant5$ for prime $p$. The exponent of the Sylow $p$-subgroup of $T$ is $p$, and the exponent of the Sylow $p$-subgroup of $\Out(T)$ is $f_p$, the $p$-part of $f$. Hence every $p$-element in $\overline{R}\lesssim\Hol(T)^k.\Sy_k$ has order at most $pf_pk$, and so the Sylow $p$-subgroup of $\overline{R}$ has order at most $(pf_pk)^2$. Thereby we conclude from (\ref{eq3}) and the fact $|T|_p^k=p^{fk}$ that
\begin{equation}\label{eq6}
p^{fk}\leqslant(pf_pk)^2.
\end{equation}
If $(k,f)=(1,1)$, then the lemma already holds. To finish the proof, we exclude the other values of $(k,f)$.

First suppose that $(k,f)=(1,2)$. Then $p\geqslant3$, $N^\Del\cong T=\PSL_2(p^2)$ and $M_\omega\cong M/N^\Del\lesssim\Aut(T)=\PSL_2(p^2).\Z_2^2$. Since $M_\omega\cong\overline{R}/\overline{R}\cap N^\Del$, we deduce from (\ref{eq3}) that $|T|$ divides $|\overline{R}|=|\overline{R}\cap N^\Del||M_\omega|$, and hence
\begin{equation}\label{eq7}
4|\overline{R}\cap N^\Del||M_\omega\cap T|\text{ is divisible by }|\PSL_2(p^2)|.
\end{equation}
Computation in \magma~\cite{magma} shows that the maximal order of metacyclic subgroups in $\Hol(\PSL_2(9))$ is $180$, contrary to (\ref{eq3}). Thus we have $p>3$. Then by the classification of subgroups of the two-dimensional projective special linear group (see for example \cite{huppert1967endliche}), we know that every subgroup of $\PSL_2(p^2)$ with order divisible by $p$ has order dividing $p^2(p^2-1)$. If both $|\overline{R}\cap N^\Del|$ and $|M_\omega\cap T|$ are divisible by $p$, then (\ref{eq7}) implies that $|\PSL_2(p^2)|$ divides $(p^2(p^2-1))^2$, which turns out to be a contradiction to that $p^2+1$ divides $8(p^2-1)$. Accordingly, exactly one of $|\overline{R}\cap N^\Del|$ and $|M_\omega\cap T|$ is divisible by $p$, say, $|\overline{R}\cap N^\Del|$. Then $|\overline{R}\cap N^\Del|$ is divisible by $p^2$ due to (\ref{eq7}), and it follows that $\Z_p^2\leqslant\overline{R}\cap N^\Del\leqslant\Z_p^2{:}\Z_{(p^2-1)/2}$. Note that the order of an element in $\Z_p^2{:}\Z_{(p^2-1)/2}\leqslant\PSL_2(p^2)$ is either coprime to $p$ or equal to $p$. We conclude $\overline{R}\cap N^\Del=\Z_p^2$ since $\overline{R}\cap N^\Del$ is metacyclic. Therefore, $(\ref{eq7})$ turns out to be $8|M_\omega\cap T|\equiv0\pmod{p^4-1}$. However, inspecting the orders of subgroups of $\PSL_2(p^2)$ we conclude that this is not possible.

Next suppose that $k=1$ and $f\geqslant3$. The exponent of the Sylow $p$-subgroup of $T$ is $p$, and the order of the Sylow $p$-subgroup of $\Out(T)$ divides $f$. Hence we deduce from $\overline{R}\lesssim\Hol(T)=T^2.\Out(T)$ that the Sylow $p$-subgroup of $\overline{R}$ has order dividing $p^2f$, and thus (\ref{eq3}) implies
\begin{equation}\label{eq8}
p^f\di p^2f.
\end{equation}
In particular, $p^3/3\leqslant p^f/f\leqslant p^2$. Therefore, $p=2$ or $3$. Substituting this into (\ref{eq8}) we deduce that $(p,f)=(2,4)$ or $(3,3)$. Computation in \magma~\cite{magma} shows that there are no metacyclic transitive subgroup in $\Hol(\PSL_2(2^4))$. Thus $(p,f)=(3,3)$, and $\Hol(T)=T{:}\Aut(T)=\PSL_2(3^3){:}(\PSL_2(3^3).\Z_6)$. Since the maximal order of metacyclic subgroups in $\PSL_2(3^3)$ is $28$ according to computation in \magma~\cite{magma}, we conclude that a metacyclic subgroup of $\Hol(T)$ has order at most $28^2\cdot6=4704<9828=|T|$, violating (\ref{eq3}).

Now suppose that $k=2$. The exponent of the Sylow $p$-subgroup of $T$ is $p$, and the order of the Sylow $p$-subgroup of $\Out(T)$ divides $f$. Hence the exponent of the Sylow $p$-subgroup of $\Hol(T)$ divides $pf$, and the same holds for the exponent of the Sylow $p$-subgroup of $\Hol(T)^2$. Since $\overline{R}\lesssim\Hol(T)^2.\Sy_2$, we deduce from (\ref{eq3}) that
\begin{equation}\label{eq6}
p^{2f}\di2(pf)^2.
\end{equation}
If $p=2$, then the assumption $p^f\geqslant5$ indicates $f\geqslant3$, under which (\ref{eq6}) has no solution. Hence we have $p\geqslant3$. If $f\geqslant2$, then (\ref{eq6}) implies that $p^4/4\leqslant p^{2f}/f^2\leqslant2p^2$, a contradiction. Consequently, $f=1$, and so $p\geqslant5$. For $p=5$ or $7$, computation in \magma~\cite{magma} shows that $\Hol(\PSL_2(p)^2)$ has no metacyclic transitive subgroup. Thus we conclude that $q=p\geqslant11$. Let $K=\Soc(\Hol(N^\Del))=T_1\times T_2\times T_3\times T_4$ and $S=\overline{R}\cap K$, where $T_i\cong T=\PSL_2(p)$ for $1\leqslant i\leqslant4$. Since $|\overline{R}/S|$ divides $8$, we deduce from (\ref{eq3}) that
\begin{equation}\label{eq11}
8|S|\text{ is divisible by }p^2(p^2-1)^2/4.
\end{equation}
For distinct $i$ and $j$ in $\{1,2,3,4\}$, denote by $\pi_i$ the projection from $K$ to $T_i$ and $\pi_{i,j}$ the projection from $K$ to $T_i\times T_j$. Since $S$ is metacyclic, one can write $S=\langle g,h\rangle$ with $\langle g\rangle\unlhd S$. As $p^2$ divides $|S|$, $p^2$ must divide $o(g)o(h)$. It follows that both $o(g)$ and $o(h)$ are divisible by $p$ since the Sylow $p$-subgroup of $K$ is $\Z_p$. Without loss of generality, assume that $p\di o(\pi_1(h))$. Then $\pi_1(S)=\Z_p$ since $\langle\pi_1(g)\rangle\trianglelefteq\pi_1(S)=\langle\pi_1(g),\pi_1(h)\rangle$. As $p^2$ divides $\prod_{i=1}^4|\pi_i(S)|$, we know that at least one of $|\pi_2(S)|$, $|\pi_3(S)|$ and $|\pi_4(S)|$ is divisible by $p$, say $|\pi_2(S)|$. Then $\pi_2(S)\leqslant\Z_p{:}\Z_{(p-1)/2}$, and thereby we conclude from (\ref{eq11}) that
\begin{equation}\label{eq12}
|\pi_{3,4}(S)|\text{ is divisible by }(p-1)(p+1)^2/16.
\end{equation}
For $p=11$, $13$ or $17$, computation in \magma~\cite{magma} shows that there is no metacyclic subgroup in $\PSL_2(p)^2$ of order divisible by $(p-1)(p+1)^2/16$, contradicting (\ref{eq12}). Thus $p\geqslant19$. Now since the $p'$-part of any metacyclic subgroup in $\PSL_2(p)$ is at most $p+1$, we have $|\pi_3(S)|_{p'}|\pi_4(S)|_{p'}\leqslant(p+1)^2<(p-1)(p+1)^2/16$, contradicting (\ref{eq12}) again.

Finally suppose that $k\geqslant3$. The exponent of the Sylow $p$-subgroup of $T$ is $p$, and the exponent of the Sylow $p$-subgroup of $\Out(T)$ divides $f$. Moreover, the order of a $p$-element in $\Sy_k$ divides $p^{\lfloor\log_pk\rfloor}$. Hence each $p$-element in $\overline{R}\lesssim\Hol(T)^k.\Sy_k$ has order dividing $pfp^{\lfloor\log_pk\rfloor}$, and so the Sylow $p$-subgroup of $\overline{R}$ has order at most $(pfp^{\lfloor\log_pk\rfloor})^2$. Thereby we conclude from (\ref{eq3}) and the fact $|T|_p^k=p^{fk}$ that
\begin{equation}\label{eq9}
p^{fk}\di(pfp^{\lfloor\log_pk\rfloor})^2.
\end{equation}
As a consequence,
\begin{equation}\label{eq10}
p^{fk}\leqslant(pfk)^2,
\end{equation}
and so $p^3/9\leqslant p^{fk}/(fk)^2\leqslant p^2$. This gives $p\in\{2,3,5,7\}$. If $p=2$, then the assumption $p^f\geqslant5$ indicates $f\geqslant3$, contrary to (\ref{eq10}). In the same way $p=3$ causes a contradiction. If $p=5$ or $7$ then we have $f=1$ by (\ref{eq10}), but (\ref{eq9}) has no solution for $k$ if $(p,f)=(5,1)$ or $(7,1)$. Thus the proof is completed.
\end{proof}

We remark that, if $G$ is a quasiprimitive permutation group containing a transitive metacyclic subgroup, then
Lemma~$\ref{NormalSemireg}$ in particular shows that $G$ is not of type twisted wreath or holomorph compound.

\begin{lemma}\label{LemmaSD}
Let $G\leqslant\Sym(\Ome)$ and $R$ be a metacyclic transitive subgroup of $G$. If $G$ is quasiprimitive of simple diagonal type with socle $K=T_1\times\dots\times T_k$, where $k\geqslant2$ and $T_1\cong\dots\cong T_k\cong T$ is nonabelian simple, then $k=2$.
\end{lemma}

\begin{proof}
Note that $G$ acts on $\calT=\{T_1,\dots,T_k\}$ by conjugation, $|\Ome|=|T|^{k-1}$ and $G/K\leqslant\Out(T)\times\Sy_k$. By Lemma \ref{DirectProduct}, $|R\cap K|$ divides $|T|^2$, whence $|R|$ divides $k!|\Out(T)||T|^2$. Since $|R|$ is divisible by $|\Ome|$, we conclude that
\begin{equation}\label{eq13}
k!|\Out(T)||T|^2\text{ is divisible by }|T|^{k-1}.
\end{equation}
By~\cite[Corollary~6]{LPS-subgps}, there is a prime $p\geqslant5$ such that $p$ divides $|T|$ but does not divide $|\Out(T)|$. Hence it follows from (\ref{eq13}) that $|T|_p^{k-1}$ divides $(k!)_p|T|_p^2$. Since
\[
\log_p(k!)_p=\sum\limits_{i=1}^\infty\left\lfloor\frac{k}{p^i}\right\rfloor<\sum\limits_{i=1}^\infty\frac{k}{p^i}=\frac{k}{p-1},
\]
this implies that $\log_p|T|_p^{k-3}<k/(p-1)$. We thereby obtain $k-3<k/(p-1)\leqslant k/4$, and thus $k\leqslant3$. To complete the proof, it suffices to exclude the case $k=3$.

Suppose for a contradiction that $k=3$. If the action of $R$ on $\calT$ by conjugation fixes some point, say $T_1$, then $T_2\times T_3$ is normal in $KR$ and semiregular on $\Ome$, which is contrary to Lemma \ref{NormalSemireg}. Hence $R$ induces a transitive action on $\calT$ by conjugation, and so $R\cap T_1\cong R\cap T_2\cong R\cap T_3$. Since $(R\cap T_1)\times(R\cap T_2)\times(R\cap T_3)\leqslant R$ and $R$ is metacyclic, we conclude that $R\cap T_1=1$. Therefore, $R\cap K\cong(R\cap K)T_1/T_1\leqslant K/T_1$.
As a consequence, $|T|=|T_1|$ divides $|K|/|R\cap K|=|KR/K|$ and thus divides $|\Out(T)\times\Sy_3|=6|\Out(T)|$.
This is impossible since $p\geqslant5$ divides $|T|$ but does not divide $|\Out(T)|$.
\end{proof}

\begin{theorem}\label{Diagonal}
Let $G\leqslant\Sym(\Ome)$, and $R$ be a metacyclic transitive subgroup of $G$. Suppose that $G$ is quasiprimitive of diagonal type. Then $G$ is primitive, $\Soc(G)=\PSL_2(p)^2$ for some prime $p\geqslant5$, and $R=\Z_{p(p+1)/2}.\Z_{p-1}$. Further, if $p\equiv1\pmod{4}$ then $\Hol(\PSL_2(p))\leqslant G\leqslant\PSL_2(p)^2.\Z_2^2$; if $p\equiv3\pmod{4}$ then $\PSL_2(p)^2\leqslant G\leqslant\PSL_2(p)^2.\Z_2^2$.
\end{theorem}

\begin{proof}
Write $K=\Soc(G)$. By Lemmas \ref{NormalSemireg} and \ref{LemmaSD}, $K=T_1\times T_2$ with $T_1\cong T_2=\PSL_2(p)$ for prime $p\geqslant5$, and $|\Ome|=|\PSL_2(p)|$. This implies that $G$ is primitive, as desired. Let $H$ be the diagonal subgroup of $K$, which is a point stabilizer of $K$, $S=R\cap K$, and $\pi_i$ be the projection from $K$ to $T_i$ for $i=1,2$. Notice that $|R|$ is divisible by $|\PSL_2(p)|$, and $|R|/|S|=|KR|/|K|$ divides $4$. We deduce that $|S|$ is divisible by $|\PSL_2(p)|/4=p(p^2-1)/8$. Hence $|\pi_1(S)||\pi_2(S)|$ is divisible by $p(p^2-1)/8$ as $S\leqslant\pi_1(S)\times\pi_2(S)$.

Suppose that $T_1^g=T_2$ for some $g\in R$. Then $\pi_1(S)\cong ST_2/T_2=S^gT_1^g/T_1^g\cong ST_1/T_1\cong\pi_2(S)$, and hence $|\pi_1(S)|^2$ is divisible by $p(p^2-1)/8$. In particular, $|\pi_1(S)|$ is divisible by $p$. Since $\pi_1(S)$ is a metacyclic subgroup of $\PSL_2(p)$, we conclude that $|\pi_1(S)|$ divides $p(p-1)/2$. This causes $p(p^2-1)/8\di p^2(p-1)^2/4$, and so $p+1\di2(p-1)$, which is not possible since $p\geqslant5$. Accordingly, the action of $R$ by conjugation is not transitive on $\{T_1,T_2\}$. This means that $R$ normalizes both $T_1$ and $T_2$. As a consequence, $R\leqslant\Hol(T_1)$, and so $|R|/|S|=|KR|/|K|\leqslant|\Hol(T_1)|/|K|=2$. This together with $|R|\equiv0\pmod{|\PSL_2(p)|}$ yields that $|S|$ is divisible by $|\PSL_2(p)|/2=p(p^2-1)/4$. Hence $|\pi_1(S)||\pi_2(S)|$ is divisible by $p(p^2-1)/4$.

Let $\Pa_1=\Z_p{:}\Z_{(p-1)/2}$ be a maximal parabolic subgroup of $\PSL_2(p)$. Since $\pi_1(S)$ and $\pi_2(S)$ are metacyclic subgroups of $\PSL_2(p)$ with the product of their orders divisible by $p(p^2-1)/4$, one concludes that $\pi_i(S)\leqslant\Pa_1$ and $\pi_{3-i}(S)\leqslant\D_{p+1}$ with $i=1$ or $2$. Now $S$ is a subgroup of $\Pa_1\times\D_{p+1}$ of index at most $2$ as $|S|$ is divisible by $p(p^2-1)/4=|\Pa_1||\D_{p+1}|/2$.

First assume $p\equiv1\pmod{4}$. In this case, $\Pa_1\times\D_{p+1}$ is not metacyclic. Hence $S$ is a subgroup of index two in $\Pa_1\times\D_{p+1}$, $|R|/|S|=2$ and $R\nleqslant\Pa_1\times\D_{p+1}$. As a consequence, $R\nleqslant K$, and then $G\geqslant KR=\Hol(\PSL_2(p))$. Since $S$ has index $2$ in $\Pa_1\times\D_{p+1}$, we know that $S=((\Z_p{:}\Z_{(p-1)/4})\times\Z_{(p+1)/2}).\Z_2$. Therefore,
\[
R=S.\Z_2=((\Z_p{:}\Z_{(p-1)/4})\times\Z_{(p+1)/2}).\Z_4=\Z_{p(p+1)/2}{:}\Z_{p-1}.
\]

Next assume $p\equiv3\pmod{4}$. Then it follows from $\gcd(|\Pa_1|,|\D_{p+1}|)=1$ that $(\Pa_1\times\D_{p+1})\cap H=1$, and so $\Pa_1\times\D_{p+1}$ is transitive on $\Ome$. Moreover, $\Pa_1\times\D_{p+1}=\Z_{p(p+1)/2}.\Z_{p-1}$ is metacyclic. If $R=S$, then since $|\Pa_1\times\D_{p+1}|=|\Ome|$ divides $|R|$ and $R\leqslant\Pa_1\times\D_{p+1}$, we have $R=\Pa_1\times\D_{p+1}=\Z_{p(p+1)/2}.\Z_{p-1}$ as the lemma asserts. Assume that $|R|/|S|=2$ hereafter. If $S=\Pa_1\times\D_{p+1}$, then $R=(\Pa_1\times\D_{p+1}){:}\Z_2$ is not metacyclic, a contradiction. Hence $S$ has index $2$ in $\Pa_1\times\D_{p+1}$, which implies that $S=\Pa_1\times\Z_{(p+1)/2}$ or $\Pa_1\times\D_{(p+1)/2}$. Consequently, $R=S.\Z_2=\Z_{p(p+1)/2}.\Z_{p-1}$ as the lemma asserts.
\end{proof}

\section{Proof of Theorem \ref{Quasiprimitive}}

In this section we will complete the proof of Theorem~\ref{Quasiprimitive}.

\begin{lemma}\label{LemmaPA}
Let $H$ be a permutation group on $n$ points with $n\geqslant5$, and $G=H\wr\Sy_k$ be a primitive wreath product with base group $K=H_1\times\dots\times H_k$, where $H_1\cong\dots\cong H_k\cong H$, and
suppose that $G$ has a metacyclic transitive subgroup $R$. Then $k\leqslant2$.
\end{lemma}

\begin{proof}
Since $R$ is metacyclic, one can write $R=ML$ with $M=\Z_m$ normal in $R$ and $L=\Z_\ell$. Denote $\C_R(M)\cap L$, the centralizer of $M$ in $L$, by $C$. Note that $C$ is in the center of $R$. Thus, $M\cap K$ and $C\cap K$ are both cyclic normal subgroups of $R$, which we write as $\langle a\rangle$ and $\langle b\rangle$, respectively. By Lemma \ref{NormalCyclic}, both $\langle a\rangle$ and $\langle b\rangle$ are semiregular, and so we have $o(a)\di n$ and $o(b)\di n$ according to Lemma \ref{SemiregCyclic}. Let $\overline{R}=RK/K$, $\overline{M}=MK/K$, $\overline{C}=CK/K$ and $\overline{L}=LK/K$. Then $\overline{M}$ and $\overline{C}$ are both cyclic normal subgroups of $\overline{R}$.

\textbf{Claim.} The conclusion of the lemma holds if $\overline{R}$ is transitive on $\{1,2,\dots,k\}$.

In fact, the transitivity of $\overline{R}$ implies that $\overline{M}$ is semiregular by Lemma \ref{NormalCyclic}, and hence $|\overline{M}|$ divides $k$. Therefore, $m=|M\cap K||\overline{M}|=o(a)|\overline{M}|$ divides $nk$. For the similar reason $|C|$ divides $nk$. Then we deduce from the observation $L/C\lesssim\Aut(M)$ that $\ell$ divides $nk\phi(nk)$, where $\phi$ is the Euler's totient function. Since $|R|$ divides $m\ell$, it follows that $|R|$ divides $(nk)^2\phi(nk)$. As a consequence,
\begin{equation}\label{eq14}
n^k\di(nk)^2\phi(nk).
\end{equation}
Let $p$ be the largest prime divisor of $n$, and $p^r=n_p$ be the $p$-part of $n$. To prove the claim, we exclude the possibility for $k\geqslant3$ by distinguishing the following cases.

\textbf{Case 1.} $k=3$. In this case, both $m$ and $|C|$ divide $nk=3n$, and (\ref{eq14}) turns out to be $n\di9\phi(3n)$. If $p>3$, then the $p$-part of $\phi(3n)$ is $p^{r-1}$, violating the conclusion $n\di9\phi(3n)$. Consequently, $p\leqslant3$, and hence $n=2^s3^t$ with nonnegative integers $s$ and $t$. Since $m\di3n$, it follows that $m=2^e3^f$ with $e\leqslant s$ and $f\leqslant t+1$.

Suppose $e\geqslant2$. Then we have
\[
\Aut(M)=\Aut(\Z_{2^e})\times\Aut(\Z_{3^f})=\Z_{2^{e-2}}\times\Z_2\times\Z_{2\cdot3^{f-1}}.
\]
Since $L/C$ is cyclic and $L/C\lesssim\Aut(M)$, we conclude that $|L/C|$ divides $2^{e-2}3^{f-1}$. This implies that $m\ell=m|C||L/C|$ divides $(3n)\cdot(3n)\cdot(2^{s-2}3^t)=2^{3s-2}3^{3t+2}$, contrary to the condition $n^3\di m\ell$ since $R$ is a transitive subgroup of $G$.

Therefore, $e\leqslant1$, and hence $\Aut(M)=\Aut(\Z_{2^e})\times\Aut(\Z_{3^f})=\Z_{2\cdot3^{f-1}}$. Since $L/C$ is cyclic and $L/C\lesssim\Aut(M)$, we deduce that $|L/C|$ divides $2\cdot3^{f-1}$. It follows that $m\ell=m|C||L/C|$ divides $(3n)\cdot(3n)\cdot(2\cdot3^t)=2^{2s+1}3^{3t+2}$. By the condition $n^3\di m\ell$, we then have $s\leqslant1$. Now any $3$-element of $\Sy_n$ has order dividing $3^t$ since $n<3^{t+1}$. Hence any metacyclic $3$-subgroup of $K\leqslant\Sy_n\times\Sy_n\times\Sy_n$ has order dividing $3^{2t}$, and so $|R|_3=|R\cap K|_3|\overline{R}|_3$ divides $3^{2t}|\Sy_3|_3=3^{2t+1}$. Since $n^3$ divides $|R|$, this implies that $3^{3t}\di3^{2t+1}$, which means $t\leqslant1$. Since $n\geqslant5$, we obtain $s=t=1$ and then $n=2^s3^t=6$. However, computation in \magma~\cite{magma} shows that there is no metacyclic transitive subgroup of the primitive wreath product $\Sy_6\wr\Sy_3$, a contradiction.

\textbf{Case 2.} $k=4$. In this case, (\ref{eq14}) turns out to be $n^2\di16\phi(4n)$. If $p>2$, then the $p$-part of $\phi(4n)$ is $p^{r-1}$, violating the conclusion $n^2\di16\phi(4n)$. Consequently, $p=2$, and hence $n=2^s$. Now any $2$-element of $\Sy_n$ has order at most $n$. Hence any metacyclic $2$-subgroup of $K$ has order dividing $n^2=2^{2s}$ and so $|R\cap K|_2$ divides $n^2$. Since $n^4=2^{4s}$ divides $|R|$ and $|R|$ divides $|R\cap K||\Sy_4|=24|R\cap K|$, it follows that $n^4$ divides $8|R\cap K|_2$. Thereby we obtain $n^4\di8n^2$, which is impossible as $n\geqslant5$.

\textbf{Case 3.} $k=5$. In this case, (\ref{eq14}) turns out to be
\begin{equation}\label{eq15}
n^3\di25\phi(5n).
\end{equation}
In particular, $n^3\leqslant25\phi(5n)<25\cdot5n$. Therefore, $n^2<125$, and so $5\leqslant n\leqslant11$. Checking (\ref{eq15}) for $n=5,\dots,11$ we then have $n=5$. Since any $5$-element of $\Sy_5$ has order $5$, the Sylow $5$-subgroup of $R\cap K$ has order dividing $25$. However, $5^5$ divides $|R|$ and thus divides $|R\cap K||\Sy_5|=120|R\cap K|$. We deduce that $5^5$ divides $5|R\cap K|_5$, a contradiction.

\textbf{Case 4.} $k\geqslant6$. We deduce from (\ref{eq14}) that $n^k\leqslant(nk)^2\phi(nk)<(nk)^3$. Therefore,
\begin{equation}\label{eq16}
n^{k-3}<k^3.
\end{equation}
Since $n\geqslant5$, we obtain from (\ref{eq16}) that $5^{k-3}<k^3$, which implies $k\leqslant6$. Hence $k=6$, and (\ref{eq16}) turns out to be $n^3<6^3$. Consequently, $n=5$, but $(n,k)=(5,6)$ does not satisfy (\ref{eq14}), a contradiction. This proves the claim.

Now suppose that $\overline{R}$ is intransitive on $\{1,\dots,k\}$. Without loss of generality we assume that $\Delta=\{1,\dots,r\}$ is an orbit of $\overline{R}$ on $\{1,\dots,k\}$, where $r<k$. Note that each element $g$ of $R$ can be expressed as $g=(g_1,\dots,g_k)\sigma$ with $g_i\in H_i$ and $\sigma\in\Sy_k$. Define a map $\varphi:g\mapsto(g_1,\dots,g_r)\sigma|_\Delta$, where $\sigma|_\Delta$ is the restriction of $\sigma$ on $\Delta$. It is easy to see that $\varphi$ is a homomorphism and thus $R^\varphi$ is a transitive metacyclic subgroup of $H\wr\Sy_r$. Then the claim above implies $r\leqslant2$. This shows that each orbit of $\overline{R}$ has size $1$ or $2$, whence each element of $\overline{R}$ has order $1$ or $2$. Therefore, $m=o(a)|\overline{M}|$ and $|C|=o(b)|\overline{C}|$ both divide $2n$. It then follows from $L/C\lesssim\Aut(M)$ that $\ell$ divides $2n\phi(2n)$. Hence we have
\begin{equation}\label{eq17}
n^k\di(2n)^2\phi(2n)
\end{equation}
since $n^k$ divides $|R|$ and $|R|$ divides $m\ell$.

Suppose that $k\geqslant3$. Let $p$ be the largest prime divisor of $n$, and $p^r=n_p$ be the $p$-part of $n$. If $p>2$, then the $p$-part of $\phi(2n)$ is $p^{r-1}$, violating (\ref{eq17}). Consequently, $p=2$, and hence $n=2^s$. Now (\ref{eq17}) turns out to be
\[
2^{sk}\di(2^{s+1})^2\phi(2^{s+1})=2^{2s+2}\cdot2^s=2^{3s+2},
\]
from which we conclude $k=3$. Since any $2$-element of $\Sy_n$ has order dividing $n$, the Sylow $2$-subgroup of $R\cap K$ has order dividing $n^2$. However, since $n^3$ divides $|R|$ and $|R|$ divides $|R\cap K||\Sy_3|=6|R\cap K|$, it follows that $n^3$ divides $2|R\cap K|_2$. Thereby we obtain $n^3\di2n^2$, a contradiction. Hence $k\le 2$.
\end{proof}

\begin{lemma}\label{PA-type}
Suppose that $R$ is a metacyclic transitive subgroup of the primitive wreath product $H\wr\Sy_k$, where $H$ is a quasiprimitive permutation group of type almost simple or simple diagonal. Then $k=2$, and $H$ is an almost simple group satisfying part~\emph{(b)} of \emph{Theorem~\ref{Quasiprimitive}}.
\end{lemma}

\begin{proof}
By Lemma~\ref{LemmaPA} we have $k=2$. Let $H_1\times H_2$ be the base group of $H\wr\Sy_2$, where $H_1\leqslant\Sym(\Del_1)$ and $H_2\leqslant\Sym(\Del_2)$ such that $H_1\cong H_2\cong H$ and $\Omega=\Del_1\times\Del_2$. Let $R_i$ be the projection of $R\cap(H_1\times H_2)$ into $H_i$ for $i=1,2$. Then $R_i$ is metacyclic. Since $R$ is transitive on $\Del_1\times\Del_2$, $R\cap(H_1\times H_2)$ has at most two orbits on $\Del_1\times\Del_2$. It follows that at least one of $R_1$ or $R_2$, say $R_1$, is transitive, for otherwise $R\cap(H_1\times H_2)\leqslant R_1\times R_2$ would have at least four orbits on $\Del_1\times\Del_2$, a contradiction.
Hence $H_1$ is a quasiprimitive permutation group on $\Del_1$ containing a metacyclic transitive subgroup $R_1$. If $H$ is almost simple, then it satisfies part~(b) of Theorem~\ref{Quasiprimitive}. To complete the proof, we suppose for a contradiction that $H$ is of type simple diagonal.

Since $R\cap(H_1\times H_2)$ is metacyclic, one can write $R\cap(H_1\times H_2)=ML$ with $M$ cyclic normal in $R\cap(H_1\times H_2)$ and $L$ cyclic. For $i=1,2$, let $M_i$ and $L_i$ be the projection of $M$ and $L$ into $H_i$, respectively. By Theorem~\ref{Diagonal}, $\Soc(H_1)=\PSL_2(p)^2$ for some prime $p\geqslant5$, and $R_1=\Z_{p(p+1)/2}.\Z_{p-1}$. Hence $|\Del_1|=|\Del_2|=|\PSL_2(p)|$, and $M_1$ contains $\Z_p$ while $L_1$ does not. Noticing that the Sylow $p$-subgroup of $H_1\times H_2$ is $\Z_p^4$, we conclude that $|M|_p=|L|_p=p$ since $p^2=|\PSL_2(p)|_p^2$ divides $|M||L|$. Consequently, $L_2\geqslant\Z_p$.

Note that $R\leqslant(R_1\times R_2).\Z_2$ and $|R|$ is divisible by $|\PSL_2(p)|^2$. We deduce that $|R_2|$ is divisible by $|\PSL_2(p)|/2$ in view of $|R_1|=|\PSL_2(p)|$. Write $N_2=\Soc(H_2)=T_1\times T_2$ with $T_1\cong T_2\cong\PSL_2(p)$, and denote by $\pi_i$ the projection of $T_1\times T_2$ onto $T_i$ for $i=1,2$. The consequence of the previous paragraph implies that one of $\pi_1(L_2\cap N_2)$ or $\pi_2(L_2\cap N_2)$ contains $\Z_p$. Without loss of generality, assume that $\pi_1(L_2\cap N_2)\geqslant\Z_p$. As $\Z_p\leqslant\pi_1(L_2\cap N_2)$ normalizes $\pi(M_2\cap N_2)$ and $\pi_1(R_2\cap N_2)/\pi_1(M_2\cap N_2)$ is cyclic, we conclude that $\pi_1(R_2\cap N_2)=\Z_p$ by the classification of subgroups of $\PSL_2(p)$. Hence $p^2-1$ divides $8|\pi_2(R_2\cap N_2))|$ since $R_2\leqslant(\pi_1(R_2\cap N_2)\times\pi_2(R_2\cap N_2)).\Z_2^2$ and $|R_2|$ is divisible by $|\PSL_2(p)|/2$. Inspecting the subgroups of $\PSL_2(p)$ one derives that $p=5$ or $7$. For $p=5$, computation in \magma~\cite{magma} shows that $(\PSL_2(p)^2.\Z_2^2)\wr\Sy_2$ has no transitive metacyclic subgroup. For $p=7$, computation in \magma~\cite{magma} shows that there is no cyclic subgroups $M_2,L_2$ of $\PSL_2(p)^2.\Z_2^2$ such that $\Z_p\leqslant L_2$ normalizes $M_2$ and $|M_2L_2|$ is divisible by $|\PSL_2(p)|/2$. This contradiction completes the proof.
\end{proof}

\begin{theorem}\label{PA}
Let $G\leqslant\Sym(\Ome)$, and $R$ be a metacyclic transitive subgroup of $G$. Suppose that $G$ is quasiprimitive of product action type. Then $G\lesssim H\wr\Sy_2$ with $H\leqslant\Sym(\Del)$ almost simple and $\Soc(G)=\Soc(H)^2$. Moreover, for $\delta\in\Del$, one of the following holds:
\begin{itemize}
\item[(a)] $\A_n\leqslant H\leqslant\Sy_n$ and $\A_{n-1}\leqslant H_\delta\leqslant\Sy_{n-1}$.
\item[(b)] $\PSL_n(q)\leqslant H\leqslant\PGaL_n(q)$ and $q^{n-1}{:}\SL_{n-1}(q)\trianglelefteq\Soc(H)_\delta\leqslant\Pa_1$ or $\Pa_{n-1}$.
\item[(c)] $(H,H_\delta,R)$ lies in \emph{Table} $\ref{tab3}$.
\end{itemize}
\end{theorem}

\begin{table}[htbp]
\caption{}\label{tab3}
\centering
\begin{tabular}{|l|l|l|l|}
\hline
row & $H$ & $H_\delta$ & $R$\\
\hline
1 & $\PSL_2(11)$ & $\A_5$ & $\Z_{11}^2$, $\Z_{11}\wr\Sy_2$, $\Z_{11}\times(\Z_{11}{:}\Z_5)$ \\
\hline
2 & $\M_{11}$ & $\M_{10}$ & $\Z_{11}^2$, $\Z_{11}\wr\Sy_2$, $\Z_{11}\times(\Z_{11}{:}\Z_5)$ \\
\hline
3 & $\M_{23}$ & $\M_{22}$ & $\Z_{23}^2$, $\Z_{23}\wr\Sy_2$, $\Z_{23}\times(\Z_{23}{:}\Z_{11})$ \\
\hline
\end{tabular}
\end{table}

\begin{proof}
Let $\calB=\Del^k$ be a $G$-invariant partition of $\Ome$ such that $G\cong G^{\calB}\leqslant H\wr\Sy_k$, where $H\wr\Sy_k$ is a primitive wreath product with almost simple group $H\leqslant\Sym(\Del)$. Notice that $R^{\calB}$ is a metacyclic transitive subgroup of $G^{\calB}$. Then by Lemma~\ref{PA-type} we have $k=2$. As a consequence, $\Soc(G)=\Soc(H)^2$. Let $\delta\in\Del$, $N=H_1\times H_2$ be the base group of $H\wr\Sy_2$, where $H_1\leqslant\Sym(\Del)$ and $H_2\leqslant\Sym(\Del)$ such that $H_1\cong H_2\cong H$, and $R_i$ be the projection of $R^{\calB}\cap N$ into $H_i$ for $i=1,2$. Then $R_i$ is metacyclic.

Assume that $R^{\calB}\nleqslant N$. Then $R^{\calB}N/N=\Z_2$ and hence $R_1\cong R_2$. Since $R^{\calB}$ is transitive on $\Del^2$, $R^{\calB}\cap N$ has at most two orbits on $\Del^2$. It follows that at least one of $R_1$ or $R_2$ is transitive on $\Del$, for otherwise $R^{\calB}\cap N\leqslant R_1\times R_2$ would have at least four orbits on $\Del^2$. As $R_1\cong R_2$, we then deduce that $(H,R_1,H_\delta)$ and $(H,R_2,H_\delta)$ are both as described in Theorem~\ref{AS} (as $(G,A,B)$ there). Now assume that $R^{\calB}\leqslant N$. Then $R^{\calB}=R^{\calB}\cap N\leqslant R_1\times R_2$. Since $R^{\calB}$ is transitive on $\Del^2$, it follows that both $R_1$ and $R_2$ are transitive on $\Del$, and so both $(H,R_1,H_\delta)$ and $(H,R_2,H_\delta)$ are as described in Theorem~\ref{AS} (as $(G,A,B)$ there). Thus we have proved:

\textbf{Claim.} $(H,R_1,H_\delta)$ and $(H,R_2,H_\delta)$ are as described in Theorem~\ref{AS} (as $(G,A,B)$ there).

If $(H,H_\delta)$ is as described in part~(a) or~(b) of Theorem~\ref{AS}, then part~(a) or~(b) of the theorem holds. To complete the proof, by our claim we only need to deal with the case where $(H,H_\delta)$ lies in Table~\ref{tab1}, Table~\ref{tab2} or Table~\ref{tab4} (as $(G,B)$ there).

\textbf{Case 1.} $(H,H_\delta)$ lies in Table~\ref{tab1}.

For $(H,H_\delta)$ as in row~1 of Table~\ref{tab1}, viewing $\PSL_2(7)\cong\PSL_3(2)$ we derive that part~(b) of the theorem holds.

Assume that $(H,H_\delta)$ lies in row~2 of Table~\ref{tab1}. Then by the claim, $R_i\leqslant\AGL_1(7)$ for $i=1,2$ and so $R^{\calB}\cap N\leqslant\AGL_1(7)^2$. If $R^{\calB}\leqslant N$, then since $R^{\calB}$ is transitive on $\Del^2$, we have $R^{\calB}\cap N=R^{\calB}$ with order divisible by $|\Del|^2=|H|^2/|H_\delta|^2=196$. However, there is no metacyclic subgroup of $\AGL_1(7)^2$ of order divisible by $196$. If $R^{\calB}\nleqslant N$, then $R^{\calB}\cap N$ has index $2$ in $R^{\calB}$, which implies that $|R^{\calB}\cap N|$ is divisible by $|\Del|^2/2=98$ and the intersections of $R^{\calB}\cap N$ with $H_1$ and $H_2$ are isomorphic. However, there is no metacyclic subgroup of $\AGL_1(7)^2$ of order divisible by $98$ whose intersections with the two copies of $\AGL_1(7)$ are isomorphic. Thus $(H,H_\delta)$ does not lie in row~2 of Table~\ref{tab1}. Similar argument shows that $(H,H_\delta)$ does not lie in rows~7, 8 and~10 of Table~\ref{tab1}.

Assume that $(H,H_\delta)$ lies in row~3 of Table~\ref{tab1}. Then by the claim, $R_i=\Z_{11}{:}\Z_5$ for $i=1,2$ and so $R^{\calB}\cap N\leqslant(\Z_{11}{:}\Z_5)^2$. Since $R^{\calB}$ is transitive on $\Del^2$, we deduce that $2|R^{\calB}\cap N|$ is divisible by $|\Del|^2=55^2$. However, there is no such metacyclic subgroup $R^{\calB}\cap N$ of $(\Z_{11}{:}\Z_5)^2$. Thus $(H,H_\delta)$ does not lie in row~3 of Table~\ref{tab1}. In the similar way we see that $(H,H_\delta)$ does not lie in rows~5--6 and~11--16 of Table~\ref{tab1}.

Assume that $(H,H_\delta)$ lies in row~4 of Table~\ref{tab1}. Then searching in \magma~\cite{magma} for metacyclic transitive subgroups of $H\wr\Sy_2$ shows that $R\cong R^{\calB}$ lies in row~1 of Table~\ref{tab3}.

Assume that $(H,H_\delta)$ lies in row~9 of Table~\ref{tab1}. Then by the claim, $R_i=\Z_{19}{:}\Z_9$ for $i=1,2$ and so $R^{\calB}\cap N\leqslant(\Z_{19}{:}\Z_9)^2$. Since $R^{\calB}$ is transitive on $\Del^2$, we deduce that $2|R^{\calB}\cap N|$ is divisible by $|\Del|^2=57^2$. However, there is no such metacyclic subgroup $R^{\calB}\cap N$ of $(\Z_{19}{:}\Z_9)^2$ whose projections into the two copies of $\Z_{19}{:}\Z_9$ are both surjective. Thus $(H,H_\delta)$ does not lie in row~9 of Table~\ref{tab1}.

\textbf{Case 2.} $(H,H_\delta)$ lies in Table~\ref{tab2}.

Assume that $(H,H_\delta)$ lies in row~1 of Table~\ref{tab2}. Then by the claim, $R_i=\Z_{13}{:}\Z_3$ for $i=1,2$ and so $R^{\calB}\cap N\leqslant(\Z_{13}{:}\Z_3)^2$. Since $R^{\calB}$ is transitive on $\Del^2$, we deduce that $2|R^{\calB}\cap N|$ is divisible by $|\Del|^2=39^2$. However, there is no such metacyclic subgroup $R^{\calB}\cap N$ of $(\Z_{13}{:}\Z_3)^2$. Thus $(H,H_\delta)$ does not lie in row~1 of Table~\ref{tab2}. In the similar way we see that $(H,H_\delta)$ does not lie in rows~2--7 of Table~\ref{tab2}.

\textbf{Case 3.} $(H,H_\delta)$ lies in Table~\ref{tab4}.

Assume that $(H,H_\delta)$ lies in rows~1--3 of Table~\ref{tab4}. Then $H\leqslant\Sy_p$, $|\Delta|=p(p-1)/2$ or $p(p-1)$, and $R_i=\Z_p{:}\Z_{(p-1)/2}$ or $\Z_p{:}\Z_{p-1}$ for $i=1,2$. Since $R^{\calB}\cap N$ is metacyclic, we have $R^{\calB}\cap N=ML$ for a cyclic normal subgroup $M$ and a cyclic subgroup $L$. This implies that $R_1=(MH_2/H_2)(LH_2/H_2)$ with a cyclic normal subgroup $MH_2/H_2$ and a cyclic subgroup $LH_2/H_2$. As $\Z_p{:}\Z_{(p-1)/2}\leqslant R_1\leqslant\AGL_1(p)$, we then deduce that $MH_2/H_2=\Z_p$ and $LH_2/H_2=\Z_{(p-1)/2}$ or $\Z_{p-1}$. Similarly, $MH_1/H_1=\Z_p$ and $LH_1/H_1=\Z_{(p-1)/2}$ or $\Z_{p-1}$. As a consequence, $M\leqslant\Z_p^2$ and $L\leqslant\Z_{p-1}^2$. Hence $M=\Z_p$ and $|L|$ is not divisible by $p$. It follows that $|M||L|$ is not divisible by $p^2$ and so $|R^{\calB}\cap N|$ is not divisible by $p^2$. However, the transitivity of $R^{\calB}$ on $\Del^2$ implies that $|R^{\calB}|$ is divisible by $p^2$ and thus $|R^{\calB}\cap N|$ is divisible by $p^2$, a contradiction.

Note that $\A_5\cong\PSL_2(5)$, $\A_6\cong\PSL_2(9)$ and $\A_8\cong\PSL_4(2)$. Then for $(H,H_\delta)$ as in rows~4--6 and~8--9 of Table~\ref{tab4}, part~(b) of the theorem holds.

Assume that $(H,H_\delta)$ lies in row~7 of Table~\ref{tab4}. Then by the claim, $R_i\leqslant\Z_{10}{:}\Z_4$ for $i=1,2$ and so $R^{\calB}\cap N\leqslant(\Z_{10}{:}\Z_4)^2$. If $R^{\calB}\leqslant N$, then since $R^{\calB}$ is transitive on $\Del^2$, we have $R^{\calB}\cap N=R^{\calB}$ with order divisible by $|\Del|^2=|H|^2/|H_\delta|^2=400$. However, there is no metacyclic subgroup of $(\Z_{10}{:}\Z_4)^2$ of order divisible by $400$. If $R^{\calB}\nleqslant N$, then $R^{\calB}\cap N$ has index $2$ in $R^{\calB}$, which implies that $|R^{\calB}\cap N|$ is divisible by $|\Del|^2/2=200$ and the intersections of $R^{\calB}\cap N$ with $H_1$ and $H_2$ are isomorphic. However, there is no metacyclic subgroup of $(\Z_{10}{:}\Z_4)^2$ of order divisible by $200$ whose intersections with the two copies of $\Z_{10}{:}\Z_4$ are isomorphic. Thus $(H,H_\delta)$ does not lie in row~7 of Table~\ref{tab4}. Similar argument shows that $(H,H_\delta)$ does not lie in row~23 of Table~\ref{tab4}.

Assume that $(H,H_\delta)$ lies in row~10 of Table~\ref{tab4}. Then by the claim, $R_i=\D_{30}$, $\Z_5\times\Sy_3$ or $\AGL_1(5)\times\Z_3$ for $i=1,2$. In particular, $R_i\leqslant\AGL_1(5)\times\Sy_3$ for $i=1,2$. If $R^{\calB}\leqslant N$, then since $R^{\calB}$ is transitive on $\Del^2$, we have $R^{\calB}\cap N=R^{\calB}$ with order divisible by $|\Del|^2=|H|^2/|H_\delta|^2=900$. However, computation in \magma~\cite{magma} shows that metacyclic subgroups of $(\AGL_1(5)\times\Sy_3)^2$ of order divisible by $900$ must have order $900$, but $H^2$ does not have a metacyclic regular subgroup, a contradiction. If $R^{\calB}\nleqslant N$, then $R^{\calB}\cap N$ has index $2$ in $R^{\calB}$, which implies that $|R^{\calB}\cap N|$ is divisible by $|\Del|^2/2=450$ and the intersections of $R^{\calB}\cap N$ with $H_1$ and $H_2$ are isomorphic. However, there is no metacyclic subgroup of $\D_{30}^2$, $(\Z_5\times\Sy_3)^2$ or $(\AGL_1(5)\times\Z_3)^2$ of order divisible by $450$ whose intersections with the two factors are isomorphic. Thus $(H,H_\delta)$ does not lie in row~10 of Table~\ref{tab4}.

Assume that $(H,H_\delta)$ lies in row~11 of Table~\ref{tab4}. Then by the claim, $R_i=\Z_{31}{:}\Z_5$ for $i=1,2$ and so $R^{\calB}\cap N\leqslant(\Z_{31}{:}\Z_5)^2$. Since $R^{\calB}$ is transitive on $\Del^2$, we deduce that $2|R^{\calB}\cap N|$ is divisible by $|\Del|^2=155^2$. However, there is no such metacyclic subgroup $R^{\calB}\cap N$ of $(\Z_{31}{:}\Z_5)^2$. Thus $(H,H_\delta)$ does not lie in row~11 of Table~\ref{tab4}. In the similar way we see that $(H,H_\delta)$ does not lie in rows~12--19, 21--22 and~26 of Table~\ref{tab4}.

Assume that $(H,H_\delta)$ lies in row~20 of Table~\ref{tab4}. Then searching in \magma~\cite{magma} for metacyclic transitive subgroups of $H\wr\Sy_2$ shows that $R\cong R^{\calB}$ lies in row~2 of Table~\ref{tab3}.

Assume that $(H,H_\delta)$ lies in row~24 or~25 of Table~\ref{tab4}. Then by the claim, $R_i\leqslant\Z_{23}{:}\Z_{11}$ for $i=1,2$. For $H_\delta=\PSL_3(4){:}\Z_2$ or $\Z_2^4{:}\A_7$, then in the similar way to exclude row~11 of Table~\ref{tab4} we see that it is not possible. Hence $H_\delta=\M_{22}$. It follows that $R^{\calB}\cap N$ is a metacyclic subgroup of $(\Z_{23}{:}\Z_{11})^2$ with $2|R^{\calB}\cap N|$ divisible by $|\Del|^2=23^2$. This implies that $R^{\calB}\cap N=\Z_{23}^2$ or $\Z_{23}\times(\Z_{23}{:}\Z_{11})$, and so $R\cong R^{\calB}$ lies in row~3 of Table~\ref{tab3}. The proof is thus completed.
\end{proof}

Now we are able to give a proof of Theorem~\ref{Quasiprimitive}.

\vskip0.1in
\noindent {\bf Proof of Theorem~\ref{Quasiprimitive}:}
According to the O'Nan-Scott-Praeger theorem stated in~\cite[Section 5]{q-gps}, $G$ is of one of the eight O'Nan-Scott types there.
Moreover, $G$ is not of type twisted wreath or holomorph compound by Lemma \ref{NormalSemireg}, and is not of type compound diagonal by Lemma \ref{PA-type}.

If $G$ is affine, then the situation is as in Theorem \ref{HA}.

If $G$ is almost simple, then writing
%$L=\Soc(G)$,
$A=R$ and $B=G_\omega$, we have $G=AB$.
Note that $A$ and $B$ are core-free in $G$. Thus the factorization $G=AB$ satisfies Theorem \ref{AS}.
%The transitivity of $L$ also implies that $G=LB$.

If $G$ is of type holomorph simple or simple diagonal, then by Theorem \ref{Diagonal}, part (c) of Theorem~\ref{Quasiprimitive} holds.

If $G$ is of product action type, then the situation is as in Theorem \ref{PA}. This completes the proof.
\qed

\end{document}